\documentclass[11pt]{article}
\usepackage[a4paper, total={6in, 9in}]{geometry}
\usepackage{latexsym}
\usepackage{amsmath}
\usepackage{amssymb}
\usepackage{amsthm}
\usepackage{enumerate}
\usepackage{enumitem}
\newcommand{\RN}[1]{%
  \textup{\uppercase\expandafter{\romannumeral#1}}%
}
\newcommand{\rn}[1]{%
  \textup{\lowercase\expandafter{\romannumeral#1}}%
}

\newtheorem{theorem}{Theorem}

\newtheorem{lemma}[theorem]{Lemma}

\newtheorem{corollary}[theorem]{Corollary}
\newtheorem{question}[theorem]{Question}

\newtheorem{definition}[theorem]{Definition}

\newtheorem{rmk}[theorem]{Remark}

\newtheorem{example}[theorem]{Example}

\newtheorem{prop}[theorem]{Proposition}
\usepackage{todonotes}
\usepackage{hyperref}
\usepackage{enumitem}
\makeatletter
\newcommand*{\rom}[1]{\expandafter\@slowromancap\romannumeral #1@}
\newcommand{\upperRomannumeral}[1]{\uppercase\expandafter{\romannumeral#1}}
\newcommand{\lowerromannumeral}[1]{\romannumeral#1\relax}
\hypersetup{
	backref,colorlinks=true,linkcolor=blue,citecolor=blue,urlcolor=blue,citebordercolor={0 0 1},urlbordercolor={0 0 1},linkbordercolor={0 0 1}}
 
\newcommand{\ZZ}{\mathbb{Z}}
\newcommand{\QQ}{\mathbb{Q}}
\newcommand{\CC}{\mathbb{C}}
\newcommand{\RR}{\mathbb{R}}
\newcommand{\PP}{\mathbb{P}}
\newcommand{\XX}{\mathbb{X}}
\newcommand{\dd}{\delta}

\title{Almost toric presentations of symplectic log Calabi-Yau pairs }

\author{Tian-Jun Li, Jie Min, Shengzhen Ning}

\AtEndDocument{\bigskip{\footnotesize
		\textsc{School of Mathematics, University of Minnesota, Minneapolis, MN, US} \par
		\textit{E-mail address}: \texttt{tjli@math.umn.edu} \par
		\addvspace{\medskipamount}
		\textsc{Department of Mathematics and Statistics, University of Massachusetts, Amherst, US} \par
		\textit{E-mail address}: \texttt{minxx127@umn.edu} \par
		\addvspace{\medskipamount}
		\textsc{School of Mathematics, University of Minnesota, Minneapolis, MN, US} \par
		\textit{E-mail address}: \texttt{ning0040@umn.edu} \par
		
}}
\begin{document}
\maketitle
\begin{abstract}
    It is known (\cite{symington}) that the union of fibers over elliptic singularities of an almost toric fibered (ATF) closed symplectic four-manifold forms a symplectic log Calabi-Yau (LCY) divisor. In this paper, we show the converse: any symplectic LCY divisor can be realized as the boundary divisor of an ATF. For divisors in elliptic ruled surfaces, this realization occurs over the Möbius strip; for divisors in rational surfaces, the realization occurs over the disk and becomes canonical once we choose an additional datum, called the framing, on the space of LCYs in rational surfaces. The construction for rational surfaces is achieved by considering the symplectic analogue of the toric model used in algebraic geometry (\cite{GHK}), which motivates the introduction of a new combinatorial object that we call the bitten Delzant polygon.
\end{abstract}

\tableofcontents

\section{Introduction}
The purpose of this paper is to establish a precise correspondence between two important geometric structures on closed symplectic $4$-manifolds: almost toric fibrations and symplectic log Calabi-Yau divisors. 

\subsection{From ATF to LCY}
An {\bf almost toric fibration (ATF)} on a symplectic $4$-manifold is a Lagrangian torus fibration that allows only focus-focus and elliptic singularities; it was introduced by Symington in \cite{symington} as the generalization of symplectic toric manifolds in dimension $4$. In \cite{LS}, Leung and Symington provided a classification of the diffeomorphism types of closed symplectic $4$-manifolds that admit almost toric structures. The possible underlying smooth $4$-manifolds and the topological types of their two-dimensional ATF bases are the following:
\begin{enumerate}[label=\Roman*]
    \item rational surfaces $\CC\PP^2\#n\overline{\CC\PP}^2$, $S^2\times S^2$, bases are closed disks;
    \item elliptic ruled surfaces $S^2\times T^2$, $S^2\tilde{\times}T^2$ and their blowups $(S^2\times T^2)\#n\overline{\CC\PP}^2$, bases are either closed cylinders or M\"obius strips; 
    \item  $K3$ surface, base is a sphere;
    \item Enriques surface, base is a real projective plane; 
    \item some torus bundles over torus with prescribed monodromies, bases are either torus or Klein bottle\footnote{See Remark \ref{rmk:torusbundle}.}.
\end{enumerate} 
Observe that symplectic manifolds of type \RN{3}, \RN{4}, \RN{5} in the above list are all minimal {\bf symplectic Calabi-Yau (SCY)} surfaces. In other words, their first Chern classes are torsion classes. Moreover, the ATF bases for type \RN{3}, \RN{4}, \RN{5} do not have boundaries. In fact, this phenomenon follows from \cite[Proposition 8.2]{symington}:

\begin{prop}[Symington]\label{prop:symington}
    Let $(X,\omega)$ be a closed symplectic $4$-manifold and $\pi:(X,\omega)\rightarrow B$ be an almost toric fibration over the nodal integral affine surface $B$. Then the homology class of the full boundary preimage $\pi^{-1}(\partial B)$ is Poincar\'e dual to $c_1(X,\omega)$\footnote{The argument in \cite[Proposition 8.2]{symington} involves the choice of a Lagrangian $2$-frame and viewing $c_1(X,\omega)$ as the Poincar\'{e} dual to the degeneracy locus of this frame. So, $c_1(X,\omega)$ here really means the real class rather than the integral class.}.
\end{prop}

 On the other hand, the ATF bases of type \RN{1}, \RN{2} have non-empty boundaries. In the remainder of this paper, we will concentrate on these two types of ATFs, with the aim of investigating the other three types in future work. Let us also remark that, within the classification in \cite{LS}, rational surfaces stand out as the most accessible examples, since their almost toric base diagrams are topological disks—possibly equipped with some special points representing nodal singularities—that can often be drawn in the plane after choosing branch cuts. There has been extensive use of almost toric fibrations on rational surfaces studying symplectic embeddings (\cite{brendlschlenk,CasalVianna,CGHMP-toric-staircase,Magill}), Lagrangian tori (\cite{Vianna1,viannadelpezzo,LeeOhVianna,flux}) as well as Hamiltonian dynamics (\cite{LPR1,LPR2}). For a comprehensive and accessible exposition of the subject, see the invaluable monograph \cite{evans} by Evans.

Several related notions have also been extensively studied, including symplectic toric manifolds (\cite{Delzant}), Hamiltonian $S^1$-spaces (\cite{KarshonS1}), and semitoric integrable systems (\cite{PV2,PV1}). A very nice feature of these structures is that they are classified up to isomorphism by distinct combinatorial or analytic data: Delzant polygons in the toric case, Karshon's decorated graphs for Hamiltonian $S^1$-spaces and Pelayo-Ng\d oc's five invariants for semitoric systems. These classification results reveal more than just the diffeomorphism types of the underlying manifolds. For example, it is well known (\cite{KK07}) that not every symplectic form on a rational surface supports a Hamiltonian $S^1$-action, and in particular, many do not admit a toric or semitoric structure. In fact, the symplectic forms admitting toric structures can be described explicitly for rational surfaces of small Betti numbers, as shown in Section 3.4 of \cite{Enumerate}. Compared to these notions, almost toric fibrations exhibit significantly greater flexibility, due to the wide variety of possible base diagrams. Even for $\CC\PP^2$, one can obtain infinitely many distinct base diagrams through mutations, each associated with a Markov triple. Consequently, a complete combinatorial classification of ATFs seems to be out of reach.

Based on the above discussion, the first natural question that arises is the following:

\begin{question}\label{question:1}
    Let $X$ be a closed smooth $4$-manifold in the classification list of \cite{LS}. Given a symplectic form $\omega$ on $X$, how can we determine whether $\omega$ admits an ATF structure or not?
\end{question}

   There is also a simpler variant of the above question by replacing the symplectic form with a cohomology class in $H^2(X;\RR)$ (see also Question \ref{question:Enriques}). \cite[Section 8.6]{evans} provides a prescription to find the symplectic class for a restricted class of base diagrams where the almost toric bases are $\RR^2$ with base-nodes. We will focus on closed manifolds and our first main result answers Question \ref{question:1} in the case where $X$ is a rational surface or a blowup of elliptic ruled surface. Remarkably, the condition for a symplectic form $\omega$ on $X$ to admit an ATF is surprisingly simple.
 
 \begin{theorem}[Corollary \ref{cor:ruledmain}+Corollary \ref{maincor}]\label{thm1}
   Let $(X,\omega)$ be a symplectic manifold where $X$ is diffeomorphic to either a rational surface or a blowup of elliptic ruled surface. Then $(X,\omega)$ has an ATF structure if and only if $[\omega]\cdot c_1(X,\omega)>0$.
 \end{theorem}

\begin{rmk}
    The theorem of Liu-Ohta-Ono (\cite{Liu,OhtaOno}) asserts that any closed symplectic $4$-manifold $(X,\omega)$ with $[\omega]\cdot c_1(X,\omega)>0$ must be diffeomorphic to a rational surface or blowup of ruled surface. However, since ruled surfaces of genus greater than one do not admit ATF structures, we impose the restriction on the diffeomorphism type of $X$ in Theorem \ref{thm1}.
\end{rmk} 
 
 The `only if' part in Theorem \ref{thm1} is an immediate corollary of Proposition \ref{prop:symington}. When focusing on the cases where $X$ is a rational surface or blowup of elliptic ruled surface, the base $B$ has non-empty boundary and $\pi^{-1}(\partial B)$ will be a union of symplectic submanifolds in $(X,\omega)$ so that the condition $[\omega]\cdot c_1(X,\omega)>0$ must hold. Consequently, the difficult part in Theorem \ref{thm1} is to construct an ATF whenever $[\omega]\cdot c_1(X,\omega)>0$. This construction is guided by considering another geometric object—symplectic log Calabi-Yau divisors—which we now introduce.

 A {\bf symplectic log Calabi-Yau (LCY) divisor} is a connected configuration of embedded symplectic surfaces $D=\cup D_i$ in a closed symplectic $4$-manifold $(X,\omega)$ satisfying
 \begin{itemize}
     \item the intersections between any two components are positively transverse;
     \item there is no triple intersection;
     \item $\sum[D_i]=\text{PD}(c_1(X,\omega))$.
 \end{itemize}  
 We will call $(X,\omega,D)$ a symplectic log Calabi-Yau pair, or simply a pair. It is called orthogonal if all the intersections are $\omega$-orthogonal. This notion was introduced in \cite{limak} as the symplectic counterpart of the anticanonical pair in algebraic context first systematically investigated by Looijenga in \cite{Looijenga}. We also refer to Friedman's excellent survey \cite{Friedman} as well as the work of Gross-Hacking-Keel \cite{GHK} for the mirror symmetry aspect.
 
In view of Proposition \ref{prop:symington}, ATF and LCY, as two purely symplectic notions, are related by taking the boundary divisor $\pi^{-1}(\partial B)$. This observation makes it straightforward to obtain LCY from a given ATF.

 \subsection{From LCY to ATF}
 In this paper, we will explore the opposite direction by showing that one can also construct ATF from a given LCY.

  \begin{theorem}[Theorem \ref{thm:atfruled}+Theorem \ref{thm:main}]\label{thm:intromain}
     For any orthogonal symplectic log Calabi-Yau pair $(X,\omega,D)$, there exists an almost toric fibration $\pi:(X,\omega)\rightarrow B$ such that $D=\pi^{-1}(\partial B)$.
 \end{theorem}

For rational surfaces, the statement in Theorem \ref{thm:intromain} is actually a simplified version of Theorem \ref{thm:main}. Indeed, the ATF we construct is algorithmic in nature and canonical up to a choice of framing $\mathfrak{f}$, defined in Section \ref{section:lcy}, on the set $\mathbb{LCY}$ (resp. $\mathbb{LCY}_{\geq 2}$) of the isomorphism classes of all LCY divisors in rational surfaces (resp. with $b_2^-\geq 2$). Motivated by the almost toric blowup operation, we introduce the combinatorial object called the {\bf bitten Delzant polygon} and denote the isomorphism class of them by $\mathbb{BD}$. Each bitten Delzant polygon, which naturally gives rise to a nodal integral affine disk, determines an almost toric fibred rational surface which is unique up to a fibred symplectomorphism away from the neighborhood of the nodal singular fibers (\cite[Theorem 5.2]{symington}, \cite[Theorem 8.5]{evans}). By taking the boundary divisor, there will be a well-defined map $\mathbf{B}:\mathbb{BD}\rightarrow\mathbb{LCY}$. The canonicity of our construction is then encapsulated in a section $\mathbf{R}_{\mathfrak{f}}:\mathbb{LCY}_{\geq 2}\rightarrow\mathbb{BD}$ of the surjective map $\mathbf{B}$, which depends only on the choice of framing $\mathfrak{f}$. 

Now, we explain the idea of the algorithmic construction for $\mathbf{R}_{\mathfrak{f}}$. By taking the symplectic class, $\mathbb{LCY}$ admits a period map onto the union of symplectic cones of all rational surfaces on which the diffeomorphism groups act. Choosing a fundamental domain of symplectic cones, the fibers of the period map are all finite and can even be enumerated. The method how we count the fibers in \cite{Enumerate} is by considering the {\bf reduced models} (Section \ref{section:reducedmodel}) and all the possible blowup patterns initiating from them modulo the symmetry on rational surfaces. The ambiguity of such symmetry is eliminated by the choice of the framing $\mathfrak{f}$ which enables us to obtain a unique path of blowups. In very ideal cases, a separation of the toric and non-toric blowups in this path can already tell us how to choose a Delzant polygon with triangles lying on its edges for almost toric blowups, which gives an element in $\mathbb{BD}$. In more general cases, we borrow the idea from the toric models used for holomorphic anticanonical pairs in \cite{GHK} to construct the {\bf symplectic toric model} (Section \ref{section:toricmodel}). This will be achieved by firstly performing certain birational modifications called the {\bf $\varepsilon$-replacement} (Section \ref{subsection:replace}) determined by the blowup patterns from the reduced model. For any holomorphic anticaonical pair $(Y,D)$, \cite[Proposition 1.3]{GHK} shows that one can always find some toric blowup $(\widetilde{Y},\widetilde{D})$ of $(Y,D)$ and a toric pair $(\overline{Y},\overline{D})$ as the non-toric blowdown of $(\widetilde{Y},\widetilde{D})$. We will completely forget about the complex structures and assign each topological pair a symplectic structure. The $\varepsilon$-replacement and symplectic toric model are explicit symplectic substitutions for $(\widetilde{Y},\widetilde{D})$ and $(\overline{Y},\overline{D})$ respectively. In Section \ref{section:EF21}, we will have more discussions and examples about the symplectic toric models and the toric models in the sense of \cite{GHK} and the connection to the work \cite{MN24}. 

To `mirror' the $\varepsilon$-replacement on the LCY side, we introduce a new type of surgery for nodal integral affine disks on the ATF side, which we call the {\bf full bite}. Recall that if there is an integral affine embedding from a triangle into a nodal integral affine disk such that: 
\begin{enumerate}
   \item the triangle is disjoint from all the nodes;
    \item one edge of the triangle, called the bottom triangle, lies in the $1$-stratum of the disk;
    \item the vertex not on the bottom edge, called the top vertex, lies in the interior of the disk and has integral affine distance to the bottom edge equal to the affine length of the bottom edge, 
\end{enumerate}
then one can remove the triangle, add a node at the top vertex and glue two edges of the triangle meeting at the top vertex to obtain a new nodal integral affine disk. This procedure is well-known as the ATF blowup (\cite[Section 5.4]{symington}, \cite[Section 9.1]{evans}) and corresponds to the non-toric blowup on the LCY side (see the discussion in Section \ref{section:EF21}). Now, if we modify condition (2) by letting the bottom edge have the same affine length as the disk's edge it lies on as well as requiring the preimage of the edge under ATF has self-intersection $0$, and executes the same surgery as above, then such a procedure is called a full bite. See Figure \ref{fig:fullbite}. 

 \begin{figure*}[h]
		\includegraphics*[width=\linewidth]{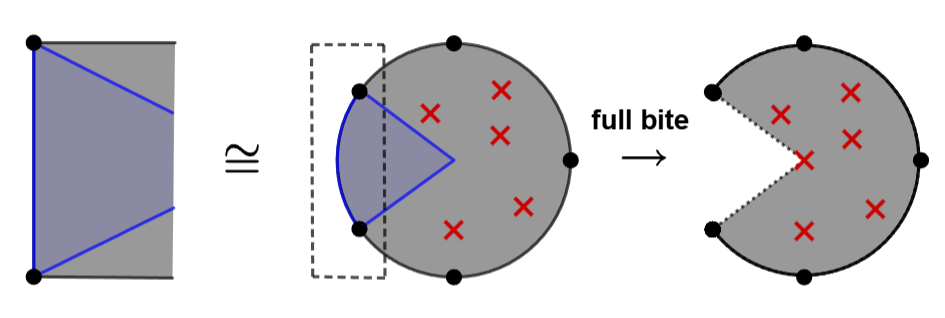}
  \caption{The middle and right diagrams represent two nodal integral affine (though not indicated in the pictures) disks related by a full bite. The full bite introduces a new node (in red) and glues the two dashed segments. In particular, two vertices of the triangle will be identified. The left diagram is an integral affine presentation of the local rectangle region in the middle diagram, where we require the edge to have self-intersection $0$. }\label{fig:fullbite}
	\end{figure*}

 The triangle used in ATF blowup or full bite will be referred to as a {\bf Symington triangle} and the integral affine length of the bottom edge will be called its {\bf size}. A full bite should be thought as the elementary transformation for algebraic surfaces \cite[Exercise \RN{3}.24]{complexagsurface}. It adds one new nodal point but also identifies two vertices of the disk so that the Euler number is not changed. A basic example is the full bite of a rectangle, which is the Delzant polygon for $S^2\times S^2$. After choosing a different cut, the diagram becomes a trapezoid with a nodal trade performed at one vertex, representing $\CC\PP^2\#\overline{\CC\PP}^2$. See Figure \ref{fig:examplefullbite} and also Section \ref{sec:bitten} for more details.  

\begin{figure*}[h]
		\includegraphics*[width=\linewidth]{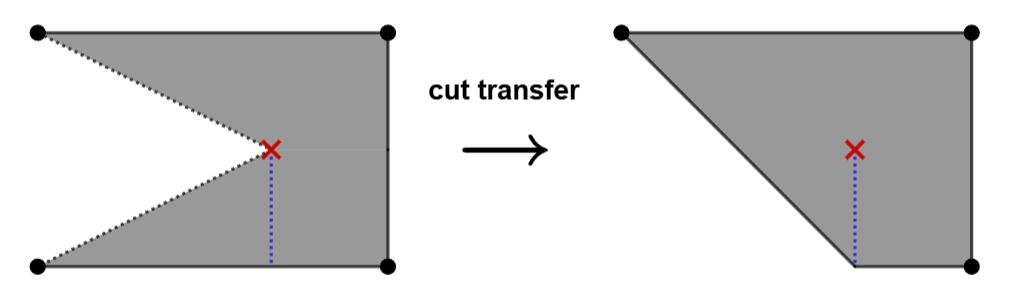}
  \caption{A full bite transforms $S^2\times S^2$ into $\CC\PP^2\#\overline{\CC\PP}^2$. }\label{fig:examplefullbite}
	\end{figure*}

\subsection{From LCY to SCY via ATF-visible symplectic sums?}
 Usher \cite{Ushersum} studied all the possible ways to obtain a symplectic Calabi–Yau surface $(X,\omega)$ through symplectic sums $(X_1,\omega_1)\#_{D_1=D_2}(X_2,\omega_2)$ along positive-genus embedded symplectic surfaces $D_1\subseteq(X_1,\omega_1),D_2\subseteq(X_2,\omega_2)$. Modulo the trivial and blowup-type sums, the two pieces $(X_i,\omega_i,D_i)$ for $i=1,2$ must be symplectic LCY pairs with exactly one divisor component. It is illuminating to compare the table in \cite{Ushersum} with the one in \cite{LS} to obtain the following table, where we use `+' to denote gluing the boundaries of two nodal integral affine surfaces in dimension $2$, and performing symplectic sums along LCY divisors in dimension $4$: 
     \begin{center}
\begin{tabular}{ |c|c| } 
 \hline
 \text{dimension }2 & \text{dimension }4 \\ 
 \hline
\text{disk+disk=sphere } & \text{rational surface+rational surface=K3 surface } \\ 
 \hline
 \text{disk+M\"obius strip=real projective plane } & \text{rational surface+ruled surface=Enriques surface} \\ 
 \hline
 \text{M\"obius strip+M\"obius strip=Klein bottle } & \text{ruled surface+ruled surface=torus bundle } \\ 
 \hline
\end{tabular}
\end{center}

\begin{rmk}\label{rmk:torusbundle}
    In the above list, the torus bundles are classified in \cite{LS} as torus bundles over the Klein bottle with Lagrangian fibers. From the symplectic sum perspective of \cite{Ushersum}, however, they can also be regarded as torus bundles over the torus, though in this setting the fibers are not Lagrangian. More precisely, these torus bundles correspond to types (e) and (f) in the table of \cite{Geiges}. Consequently, the torus bundles over the torus appearing in \cite{LS} with torus bases for their ATFs—corresponding to types (a), (b), and (d) in \cite{Geiges}—do not arise in the symplectic sum picture above.
\end{rmk}

While the ATF presentations in Theorem \ref{thm:main} are formulated for Looijenga pairs, one may perform nodal trades to smooth all nodal points and obtain embedded surfaces. This yields a rich source of ATF presentations for the summands appearing in the above table. Motivated by the results of this paper, we pose the following question as a potential application.
\begin{question}[{\bf From LCY to SCY via ATF}]\label{question:Enriques}
   Which symplectic classes on the K3 surface, the Enriques surface, and on torus bundles over the torus admit ATF structures whose base diagrams can be described as gluings of the base diagrams presenting LCYs?
\end{question}

Note that there are natural constraints on symplectic classes on a K3 surface that admit Lagrangian torus fibrations, since such a class must pair trivially with the fiber class, which is always a non-zero class. From a different perspective, via hyperK\"ahler rotation, \cite{Linyushen} establishes a density result for symplectic classes admitting ATFs. For Enriques surface and torus bundles with $b_2^+=1$, it seems reasonable to speculate that all symplectic classes should admit ATFs since their Lagrangian torus fiber classes are zero.

\hfill \break
{\bf Structure of the paper:}  We first discuss the ATF presentation for LCY in elliptic ruled surfaces in Section \ref{section:ruled}, which involves packing triangles into M\"obius strips. We then turn to the more sophisticated case of rational surfaces. Section \ref{section:lcy} provides a concise review of symplectic log Calabi-Yau divisors in rational surfaces, as studied in \cite{limak} and \cite{Enumerate}. The new ingredients are the notions of $\varepsilon$-replacement and symplectic toric model, which will play a key role later in the paper. In Section \ref{section:bd}, we introduce a class of ATF base diagrams, called the bitten Delzant polygons, realizing all LCYs, and give a careful argument showing the correspondence between these base diagrams and divisors. After addressing several technical preliminaries, we conclude in Section \ref{section:proof} with the proof of our main theorem for rational surfaces.
\hfill\break
{\bf Convention:} Throughout this paper, we use the notation $H_{i_1i_2\cdots}$ to denote the homology class $H-E_{i_1}-E_{i_2}-\cdots$.

\hfill \break
{\bf Acknowledgment:} The authors would like to thank Philip Engel, Jonny Evans, Margaret Symington and Weiwei Wu for helpful communications.

\section{Bitten M\"obius strips and ATF on elliptic ruled surfaces}\label{section:ruled}
In this section, we first construct the almost toric realization for log Calabi-Yau pair $(X,\omega,D)$ where $(X,\omega)$ is the (blowup of) elliptic ruled surface. A basic fact, which can be easily deduced from the adjunction formula (\cite[Lemma 3.1]{limak}), is that $D$ consists of a single embedded symplectic torus representing $\text{PD}(c_1(X,\omega))$. As explained in \cite{LS}, the topological types of almost toric bases for minimal elliptic ruled surfaces $S^2\times T^2$ or $S^2\tilde{\times}T^2$ could be either the cylinder $S^1\times I$ or the M\"obius strip $S^1\tilde{\times}I$\footnote{For convenience, we will always assume $I=[0,1]$ by rescaling the symplectic form.}. Since $D$ is connected, the only candidate for realizing $D$ as the ATF boundary divisor must be the M\"obius strip. In \cite[Example 5.10]{LS}, the M\"obius strip is viewed as the $\ZZ_2$-quotient of the cylinder. Here, we will adopt an alternative viewpoint which is more suitable for packing Symington triangles to realize almost toric blowups. To describe an integral affine structure on the M\"obius strip $S^1\tilde{\times}I$, we can take its universal cover $\RR\times I\subseteq \RR^2$ where the inclusion into $\RR^2$ serves as the developing map. Choose two parameters $k\in \ZZ_{\geq 0}$ and $a\in \RR_{+}$ and consider the following equivalence relation defined on $\RR\times I$ given by
\begin{align}\label{action}
    (x,y)\sim_{k,a} (x+a+ky,1-y)
    \end{align}
whose corresponding deck transformations preserve the integral affine structure on $\RR\times I$. Therefore, the quotient $\RR\times I/\sim_{k,a}$ naturally becomes an integral affine M\"obius strip. See Figure \ref{fig:strip}.

  \begin{figure}[h]
		\includegraphics*[width=\linewidth]{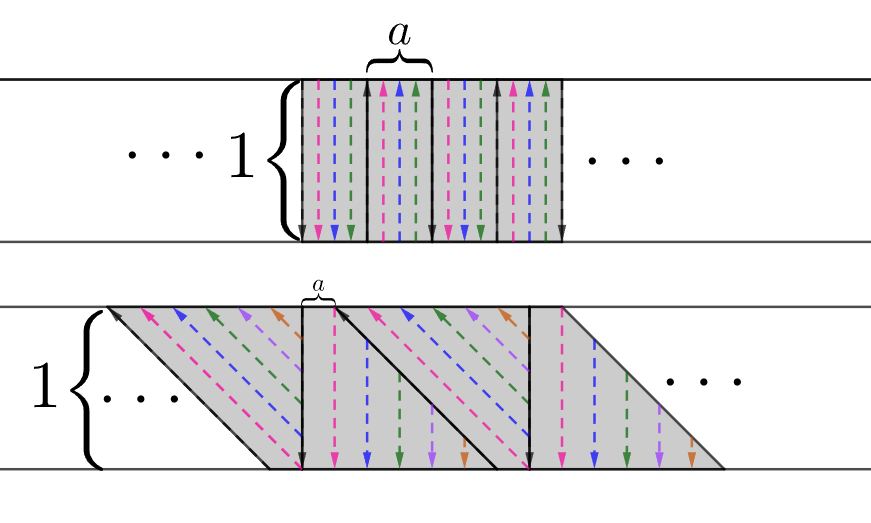}
 
		\caption{Equivalence relation (\ref{action}) on the strip $\RR\times[0,1]$ for $k=0$ and $k=1$.\label{fig:strip}}
	\end{figure}

Now, let us fix some notations. For the trivial $S^2$-bundle over $T^2$, we use $B,F\in H_2(S^2\times T^2;\ZZ)$ to denote $[\{pt\}\times T^2]$, the class of section with self-intersection $0$, and the fiber class $[S^2\times \{pt\}]$ respectively; for the non-trivial $S^2$-bundle over $T^2$, we still use $F$ to denote the fiber class but use $B_1$ for the class of section with self-intersection $1$. The following lemma gives the prescription to read off the symplectic class on minimal elliptic ruled surface from (\ref{action}). The proof relies on the ATF visible surface technique of \cite[Section 7]{symington} to find the symplectic areas of symplectic surfaces representing the fiber class and the section class.

\begin{lemma}\label{lem:minimalruled}
    The integral affine structure on the M\"obius strip induced by the quotient $\RR\times I/\sim_{k,a}$ is the almost toric base of $(S^2\times T^2,\omega)$ if $k$ is even and $(S^2\tilde{\times}T^2,\tilde{\omega})$ if $k$ is odd. Moreover, up to a rescaling, the symplectic form $\omega$ has period $\omega(F)=1$, $\omega(B)=a+\frac{k}{2}$; $\tilde{\omega}$ has period $\tilde{\omega}(F)=1$, $\tilde{\omega}(B_1)=a+\frac{k+1}{2}$. 
\end{lemma}

\begin{proof}
 Let $(p_1,p_2)\in \RR\times I$ be the action coordinates and $(q_1,q_2)$ be the angle coordinates for the Lagrangian torus fibers such that the symplectic form is locally given by $dp_1\wedge dq_1+dp_2\wedge dq_2$ around the regular fibers. For both odd and even $k$, there is a symplectic sphere lying over the blue vertical segment in Figure \ref{fig:minimalruled} by choosing $\partial_{q_2}$ as the tangent directions of the circles in the Lagrangian torus fibers collapsing to points at the boundary of the strip. Note that the class of any symplectic sphere in minimal elliptic ruled surfaces must be the fiber class $F$ by considering the projection to the $T^2$-factor. As a result, we see that $\omega(F)$ and $\tilde{\omega}(F)$ correspond to the width of our strip, which may be assumed to be $1$ up to a rescaling. On the other hand, we can take the green and red smooth curves shown in Figure \ref{fig:minimalruled} which are $C^0$-close to the diagonals in the fundamental trapezoids but have slope $0$ when approaching the boundary. By making such a choice, we will obtain two smooth curves in the M\"obius strip. They both intersect the boundary of the M\"obius strip once with the same tangent direction. Let us further choose $\partial_{q_1}$ as the tangent directions in the Lagrangian torus fibers. Note that this is not a collapsing class with respect to the boundary stratum. Indeed, it satisfies the requirement $(3)$ in \cite[Definition 7.3]{symington} so that we will get embedded symplectic tori intersecting the boundary divisor at a circle. These two tori intersect the fiber spheres positively at one point and have symplectic area $a$ and $a+k$ respectively. Therefore, the classes for red and green tori must be either $B+xF,B+(x+k)F$ or $B_1+xF,B_1+(x+k)F$ for some $x\in \ZZ$. By further observing that these two visible tori are disjoint and computing the intersection numbers:
 \[(B+xF)\cdot(B+(x+k)F)=2x+k,\]
 \[(B_1+xF)\cdot (B_1+(x+k)F)=2x+k+1,\]
 one can conclude that when $k$ is even, $x=-\frac{k}{2}$ and the classes are $B-\frac{k}{2}F,B+\frac{k}{2}F$; when $k$ is odd, $x=\frac{k+1}{2}$ and the classes are $B_1-\frac{k+1}{2}F,B_1+\frac{k-1}{2}F$. The symplectic area for $B$ or $B_1$ then immediately follows from this observation.
\end{proof}

 \begin{figure}[h]
		\includegraphics*[width=\linewidth,height=4cm]{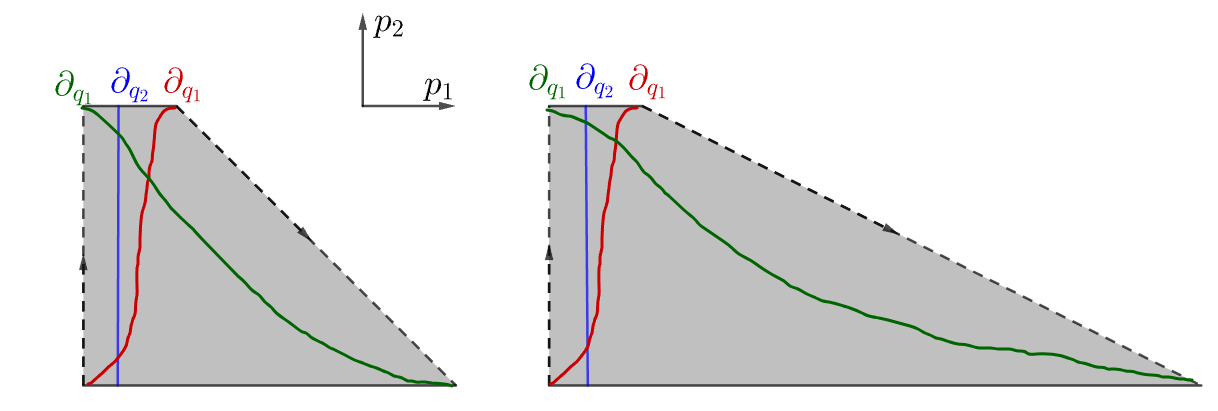}
		\caption{Visible symplectic surfaces representing the fiber class and section class, when $k=1$ and $k=2$.\label{fig:minimalruled}}
	\end{figure}
    
\begin{rmk}[Visible symplectic circle sum]
    The circle sum construction was originally introduced in \cite{LiLicirclesum} for two disjoint embedded surfaces $\Sigma_1,\Sigma_2$ in the smooth category to study the minimal genus problem. It produces a new embedded surface of genus $g(\Sigma_1)+g(\Sigma_2)-1$ by removing a neighborhood of non-separating circles on $\Sigma_1,\Sigma_2$ and concatenating them by two cylinders. This can also be formulated in the symplectic category as a special case of the tilted transport construction (\cite[Section 4]{DorfLicirclesum}). For minimal elliptic ruled surfaces, one can visualize the circle sum of two section tori lying over the red and green curves in the ATF base diagram used in the proof of Lemma \ref{lem:minimalruled}. This operation yields another embedded symplectic torus representing the class $2B$ or $2B_1-F$, which is isotopic the boundary LCY divisor. See Figure \ref{fig:circlesum}.
\end{rmk}

\begin{figure}[h]
		\includegraphics*[width=\linewidth,height=5cm]{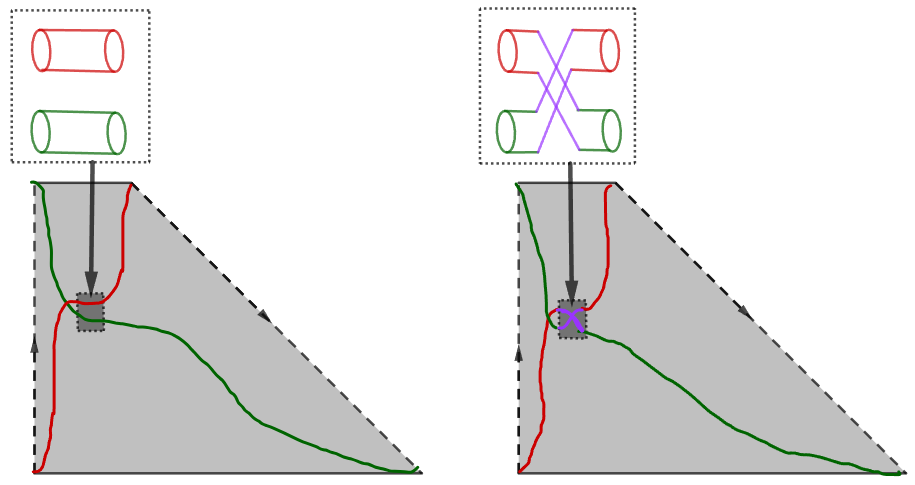}
 
		\caption{Visible symplectic circle sum.\label{fig:circlesum}}
	\end{figure}
    
Now, let us explain the version of Torelli theorem for (the blowup of) elliptic ruled surface $X$ that will be used. Suppose there are two log Calabi-Yau pairs $(X,\omega_0,D_0)$ and $(X,\omega_1,D_1)$ which satisfy $[\omega_0]=[\omega_1]$. Note that since the divisors $D_0$ and $D_1$ only contain one single component representing the Poincar\'e dual of the first Chern class. It then follows that $[D_0]=[D_1]$ by the fact that cohomologous symplectic forms on closed $4$-manifolds must have the same Chern class (\cite[Corollary A]{Salamonsurvey}). By the Torelli theorem stated in \cite[Theorem 1.4]{limak}, there is a strict symplectic deformation $(X,\omega_t,D_t)$ between them. Here, $\{\omega_t\}$ is a symplectic isotopy between $\omega$ and $\omega'$, $\{D_t\}$ is a smooth isotopy between $D_0$ and $D_1$ such that each $D_t$ is $\omega_t$-symplectic. Let $\{\phi_t\}$ be the Moser family of diffeomorphisms such that $\phi_t^*\omega_t=\omega_0$. Then we have a symplectic isotopy of $\omega_0$-embedded symplectic tori $\{\phi_t^{-1}(D_t)\}$ in $X$. By \cite[Proposition 0.2]{SiebertTian}, there is a family of Hamiltonian $\omega_0$-symplectomorphisms $\{\psi_t\}$ such that $\phi_t^{-1}(D_t)=\psi_t(D_0)$ for each $t$. Consequently, these two pairs are symplectomorphic in the sense that $(\phi_1\circ\psi_1)^*\omega_1=\omega_0,(\phi_1\circ\psi_1)(D_0)=D_1$. The above discussion can be summarized by the following version of Torelli theorem.

\begin{theorem}[{\bf Symplectic Torelli Theorem\textendash Elliptic Ruled Case}]\label{thm:torelliruled}
    Let $X$ be an elliptic ruled surface or its blowup. Then, any two log Calabi-Yau pairs $(X,\omega,D)$ and $(X,\omega',D')$ are symplectomorphic if $\omega$ and $\omega'$ are cohomologous.
\end{theorem}

Theorem \ref{thm:torelliruled} guarantees that, to realize a log Calabi-Yau divisor in (the blowup of) elliptic ruled surfaces as the boundary divisor, it suffices to construct an almost toric base diagram whose total space has the desired symplectic class. We will consider the triangle packing problem for M\"obius strip in order to keep track of the blowups, which turns out to be much simpler than that of rational surfaces since the boundary of M\"obius strip has exactly one edge and admits only non-toric (ATF) blowups.

\begin{theorem}\label{thm:atfruled}
   Let $X$ be an elliptic ruled surface or its blowup. Then, any symplectic log Calabi-Yau pair $(X,\omega,D)$ can be realized as the boundary divisor of an almost toric fibration on $(X,\omega)$. 
\end{theorem}

\begin{proof}
    By Lemma \ref{lem:minimalruled}, when $X$ is minimal, it is easy to see that we can realize all the symplectic classes by choosing suitable values of $a$ and $k$. Hence, we may assume $X$ is not minimal, in which case $X=(S^2\times T^2)\#l\overline{\CC\PP}^2$ for some $l>0$ and a standard basis of $H_2(X;\ZZ)$ can be chosen to be $\{B,F,E_1,\cdots,E_l\}$ where $E_i$ denotes the exceptional class. Note that the existence of the symplectic divisor $D$ representing the first Chern class $c_1(X,\omega)=PD(2B-\sum_{i=1}^l E_i)$ implies that 
    \begin{align}\label{equ:2}
    2\omega(B)>\sum_{i=1}^l\omega(E_i).    
    \end{align} Moreover, we also have 
    \begin{align}\label{equ:3}
    \omega(E_i)<\omega(F)     
    \end{align}for all $i$ since each class $F-E_i$ is an exceptional class and thus has non-trivial Gromov invariant.
    
    Now, let us pick a suitable fundamental domain for the quotient $\RR\times [0,1]/\sim_{0,\omega(B)}$ to create enough space for packing Symington triangles. As shown in Figure \ref{fig:ruledpacking}, we can take the isosceles triangle with height $\omega(F)$ and base length $2\omega(B)$. By (\ref{equ:2}) and (\ref{equ:3}), we need to consider $l$ Symington triangles whose sizes are all less than $\omega(F)$ and the sum of sizes are less than $2\omega(B)$. It is straightforward to see that these conditions guarantee that $l$ Symington triangles with sizes $\omega(E_1),\cdots, \omega(E_l)$ can be packed into a smaller isosceles triangle with height $\max\{\omega(E_1),\cdots,\omega(E_l)\}$ and base length $\sum_{i=1}^l\omega(E_i)$ (shown as the purple triangle in Figure \ref{fig:ruledpacking}). By choosing any disjoint embedding of these Symington triangles, one can remove their interiors from the M\"obius strip, add $l$ nodes at their top vertices and then glue together the edges that share the same top vertex. The outcome will be a new integral affine M\"obius strip with $l$ nodes, which serves as the ATF base of $(S^2\times T^2)\#l\overline{\CC\PP}^2$ equipped with some symplectic form $\omega'$ according to the diffeomorphism classification in \cite{LS}.
    
    By Torelli Theorem \ref{thm:torelliruled},  to conclude that this construction realizes the pair $(X,\omega,D)$, it remains to verify that $\omega'$ has the same period as $\omega$. To this end, we apply ATF visible surface technique again as in the proof of Lemma \ref{lem:minimalruled} for the minimal case. By perturbing the green, red and blue curves in Figure \ref{fig:minimalruled} to avoid intersecting any Symington triangle, one can still obtain two symplectic tori $T_1,T_2$ and a symplectic sphere $S$ of self-intersection zero. The same argument in the proof of Lemma \ref{lem:minimalruled} implies that $[T_1]=[T_2]=B$ and $[S]=F$. By \cite[Exercise 5.17]{symington} or \cite[Section 9.1]{evans}, there will be $l$ disjoint exceptional spheres $S_1,\cdots,S_l$ lying over the dashed segments connecting the nodes to the boundary with $\omega'$-symplectic areas $\omega(E_1),\cdots,\omega(E_l)$. Note that any exceptional class must be either $E_i$ or $F-E_i$ (\cite[Corollary 5.C]{biranpacking}). The fact that all $S_i$'s are also disjoint from $T_1,T_2$ implies that $[S_i]=E_i$ for all $i$ up to a diffeomorphism acting on $H_2(X;\ZZ)$ by permuting $E_i$'s. Therefore, we have seen that the $\omega'$-symplectic areas of $T_1,S,S_1,\cdots,S_l$ are given by $\omega(B)$, $\omega(F),\omega(E_1),\cdots,\omega(E_l)$. This implies $\omega'$ has the same period as $\omega$.
\end{proof}

 \begin{figure}[h]
		\includegraphics*[width=\linewidth,height=4cm]{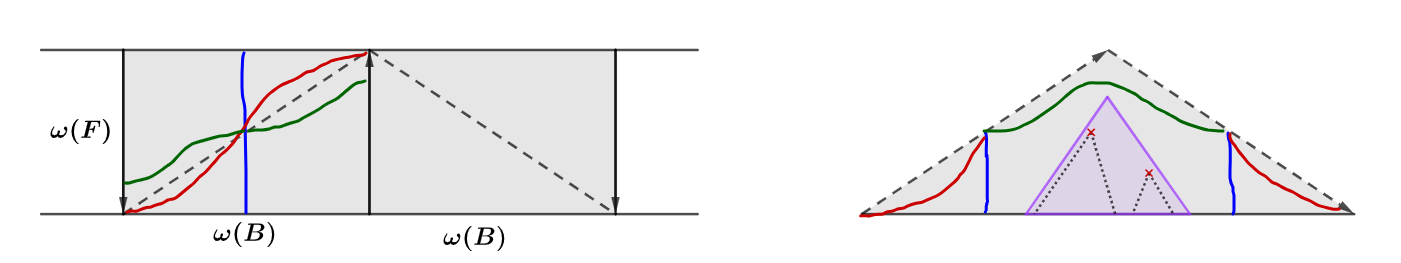}
		\caption{Choose the isosceles triangle from the left diagram as the fundamental region for the equivalence relation $\sim_{0,\omega(B)}$. In the right diagram, all Symington triangles are placed in the purple region. Over the green and red curves, there are two disjoint visible symplectic tori; over the blue curve, there is a visible symplectic sphere. \label{fig:ruledpacking}}
	\end{figure}

\begin{corollary}\label{cor:ruledmain}
    Let $X$ be an elliptic ruled surface or its blowup. If $\omega$ is a symplectic form on $X$ satisfying $[\omega]\cdot c_1(X,\omega)>0$, then $(X,\omega)$ admits an almost toric fibration.
\end{corollary}

\begin{proof}
    By Theorem \ref{thm:atfruled}, it suffices to explain the existence of an embedded symplectic tori Poincar\'e dual to $c_1(X,\omega)$. First, note that such a smooth representative $D\subseteq X$ always exists. Then, by \cite[Theorem 2.13]{relativecone}, the condition $[\omega]\cdot c_1(X,\omega)>0$ implies that the class $[\omega]$ is in the relative symplectic cone of the pair $(X,D)$. The conclusion then follows from the fact that cohomologous symplectic forms on (the blowup of) ruled surfaces are actually diffeomorphic (\cite[Theorem 3.5]{Li08}).
\end{proof}

\section{$\mathcal{LCY}$ in rational surfaces}\label{section:lcy}
\subsection{General facts and notations}
From now on, we will focus on rational surfaces, where the moduli of symplectic log Calabi–Yau divisors exhibits richer and more intricate combinatorial structures, arising from the birational relations among them. We begin by collecting some necessary background; further details can be found in \cite{limak} and \cite{Enumerate}. 

An LCY divisor $D$ in a symplectic rational surface $(X,\omega)$ is either an embedded symplectic torus as the case for elliptic ruled surfaces, which is called an {\bf elliptic pair}; or a cycle of embedded symplectic spheres whose intersections are positively transverse, which is called a {\bf Looijenga pair}. For the purpose of constructing almost toric realizations, we will restrict to pairs with $\omega$-orthogonal intersections, which can always be achieved by some Hamiltonian perturbation. A pair $(X,\omega,D)$ is called {\bf toric} if the number of components in $D$ attains the maximum $b_2(X)+2$, corresponding to the number of the edges in the Delzant polygon of the toric action on $(X,\omega)$. Two pairs $(X,\omega,D)$ and $(X',\omega',D')$ are said to be {\bf isomorphic} if there is a symplectomorphism between $(X,\omega)$ and $(X',\omega')$ which maps $D$ to $D'$. We will use the notation $[X,\omega,D]$ to denote the isomorphism class of $(X,\omega,D)$. By the {\bf homology/self-intersection/area sequence} of a pair $(X, \omega, D)$, 
where $D = \bigcup_{i=1}^n D_i$, we mean the tuple consisting of the homology classes 
$([D_1], \dots, [D_n])$, the self-intersection numbers $([D_1]^2, \dots, [D_n]^2)$, 
and the symplectic areas $(\omega(D_1), \dots, \omega(D_n))$, up to cyclic 
or anti-cyclic symmetry.

The symplectic blowup construction, which removes a Darboux ball and collapses its contact boundary via Hopf fibration, extends naturally to the pair $(X,\omega,D)$. Let $B^4(r)$ be the ball of radius $r$ in $\RR^4$ equipped with the standard symplectic form $\omega_{\text{std}}=dx_1\wedge dy_1+dx_2\wedge dy_2$. If the Darboux embedding $i:(B^4(r),\omega_{\text{std}})\hookrightarrow(X,\omega)$
\begin{itemize}
    \item intersects only one component $C$ in $D$ and $i^{-1}(D)=\{x_2=y_2=0\}$, then in the blowup symplectic manifold $(X',\omega')$, the proper transform $D'\subseteq(X',\omega')$ of $D$ will be an LCY divisor and the pair $(X',\omega',D')$ is called the {\bf non-toric blowup} of $(X,\omega,D)$;
    \item is centered at the intersection point of two components $C_1,C_2$ in $D$ and $i^{-1}(D)=\{x_1=y_1=0\}\cup \{x_2=y_2=0\}$, the total transform $D'\subseteq(X',\omega')$ of $D$ will be an LCY divisor and the pair $(X',\omega',D')$ is called the {\bf toric blowup} of $(X,\omega,D)$.
\end{itemize} 
Thus, a non-toric blowup yields an exceptional sphere transversely intersecting the proper transform of $C$; while the toric blowup produces an exceptional sphere that appears as a component in $D
'$ intersecting the proper transforms of $C_1,C_2$. They (their homology class) will be called the non-toric/toric exceptional sphere (class) respectively. When the radius $r$ is sufficiently small, such a Darboux embedding relative to $D$ always exists. We will occasionally use the notion {\bf small blowup} to refer to either non-toric or toric blowup performed with a very small size. Conversely, the symplectic blowdown operations can similarly be extended to the setting of pairs, involving either non-toric or toric exceptional spheres. Note that the orthogonal condition is preserved under the blowup/blowdown operations. 

 Let us recall the following Torelli theorem for LCY in rational surfaces from \cite[Proposition 2.11]{Enumerate}, which refines an earlier version in \cite{limak}.
\begin{theorem}[{\bf Symplectic Torelli theorem\textendash Rational Case}]\label{prop:Torelli}
    Let $(X,\omega,D)$, $(X',\omega',D')$ be two orthogonal pairs. If there is an integral isometry $\gamma:H_2(X';\ZZ)\rightarrow H_2(X;\ZZ)$ such that it maps $D'$ to $D$ componentwisely, and its real extension $\gamma_\RR:H_2(X';\RR)\rightarrow H_2(X;\RR)$ maps $PD([\omega'])$ to $PD([\omega])$, then $(X',\omega',D')$ is isomorphic to $(X,\omega,D)$.
\end{theorem}

We will frequently apply the above theorem in the following particular way: if two pairs are already known to be isomorphic and there is a correspondence among their divisor components, then\begin{itemize}
    \item performing toric blowups at corresponding nodes or non-toric blowups at corresponding components will yield new pairs that remain isomorphic if the Darboux balls have the same radius;
    \item performing toric blowdowns at corresponding toric exceptional spheres or non-toric blowdowns at non-toric exceptional spheres intersecting corresponding components will yield new pairs that remain isomorphic if the exceptional spheres have the same symplectic area.
\end{itemize} 

Theorem \ref{prop:Torelli} demonstrates that the isomorphism class of a pair is fully determined by its homological invariants. When considering only toric pair $(X,\omega,D)$ where $D=\cup C_i$, the following tautness result from \cite[Lemma 2.37]{Enumerate} shows that the isomorphism class can also be encoded by numerical invariants $[C_i]^2$ and $\omega([C_i])$.

\begin{lemma}[Tautness of toric pairs]\label{lem:torictaut}
    Two toric orthogonal Looijenga pairs with the same self-intersection sequence and symplectic area sequence up to cyclic and anti-cyclic permutation are isomorphic.
\end{lemma}

Now, let $\mathbb{LCY}$ denote the set of all the isomorphism classes of orthogonal Looijenga pairs in symplectic rational surfaces. This set can be naturally equipped with a quiver structure since different pairs are related by blowup operations: we put an arrow going from $\mathbb{X}$ to $\mathbb{X'}$ whenever there exists $(X,\omega,D)$ in class $\mathbb{X}\in\mathbb{LCY}$ and its toric or non-toric blowup $(X',\omega',D')$ is in class $\mathbb{X'}\in \mathbb{LCY}$.\footnote{Two vertices can only have at most one arrow between them.} Furthermore, we can take the free category $\mathcal{LCY}$ generated by this quiver $\mathbb{LCY}$ (in the sense of \cite[\rom{2}.7]{cat}): the objects of the category are the vertices of the quiver, the morphisms are paths between objects and the composition operation is given by concatenation of paths. For example, consider two paths of blowup operations
\[[X_0,\omega_0,D_0]\xrightarrow{f_1}\cdots\xrightarrow{f_n}[X_n,\omega_n,D_n],\,\,\, [X_n,\omega_n,D_n]\xrightarrow{f_{n+1}}\cdots\xrightarrow{f_m}[X_m,\omega_m,D_m]\]
where each $f_i$ denotes either a toric or non-toric blowup. These paths define morphisms in \[\text{Hom}([X_0,\omega_0,D_0],[X_n,\omega_n,D_n])\text{ and }\text{Hom}([X_n,\omega_n,D_n],[X_m,\omega_m,D_m]),\] which can be composed to yield a morphism in
$\text{Hom}([X_0,\omega_0,D_0],[X_m,\omega_m,D_m])$ given by
\[[X_0,\omega_0,D_0]\xrightarrow{f_1}\cdots\xrightarrow{f_n}[X_n,\omega_n,D_n]\xrightarrow{f_{n+1}}\cdots\xrightarrow{f_m}[X_m,\omega_m,D_m].\]
We will also use $\mathcal{LCY}_{\geq l}$ (resp. $\mathbb{LCY}_{\geq l}$) to denote the full subcategory of $\mathcal{LCY}$ (resp. subset of $\mathbb{LCY}$) consisting of $[X,\omega,D]$ where $X=\CC\PP^2\#n\overline{\CC\PP}^2$ with $n\geq l$. 

For any $n\in \ZZ_+$, define $$\widetilde{\Delta}_n:=\{(a,b_1,\cdots,b_n)\in\RR^{n+1}\,|\,a\geq b_1+b_2+b_3, b_1\geq b_2\geq \cdots \geq b_n>0\}\footnote{When $n=0,1,2$, $a\geq b_1+b_2+b_3$ is replaced by $a>0,a>b_1,a>b_1+b_2$ respectively.}$$ to be a polyhedral cone in $\RR^{n+1}$ and let $$ \Delta_n:=\{(1,b_1,\cdots,b_n)\in\widetilde{\Delta}_n\,|\,b_1+\cdots+b_n<3\}$$ be a slice in $\widetilde{\Delta}_n$ which is a polytope of dimension $n$. We call $\widetilde{\Delta}_n$ and $\Delta_n$ the {\bf reduced cone} and {\bf $c_1$-positive polytope} respectively. The main result in \cite{KK17} indicates that $\Delta_n$ serves as a fundamental region in
\[\{[\omega]\,|\,\omega\text{ is a symplectic form with }[\omega]\cdot c_1(\omega)>0\}\subseteq H^2(\CC\PP^2\#n\overline{\CC\PP}^2;\RR)\]
 under the diffeomorphism group action and $\RR_+$-rescaling. Let us also introduce $\Delta_{-1}:=(0,1]$ viewed as such fundamental region for $S^2\times S^2$. 
  
By a {\bf framing} on a symplectic rational manifold $(X,\omega)$ when $X=\CC\PP^2\#n\overline{\CC\PP}^2$, we mean an ordered choice of a basis $\{H,E_1,\cdots,E_n\}\subseteq H_2(X;\ZZ)$ such that
\begin{itemize}
    \item they are represented by disjoint embedded symplectic spheres in $(X,\omega)$;
    \item the first Chern class $c_1(X,\omega)$ is Poincar\'e dual to $3H-E_1-\cdots-E_n$.
\end{itemize} The framing is further said to be {\bf reduced} if the sequence of symplectic areas $(\omega(H),\omega(E_1),\cdots,\omega(E_n))$ belongs to the reduced cone $\widetilde{\Delta}_n$. When $X=S^2\times S^2$, a framing refers to an ordered choice of basis $\{B,F\}\subseteq H_2(X;\ZZ)$ such that 
\begin{itemize}
    \item they are represented by disjoint embedded symplectic spheres;
    \item the first Chern class $c_1(X,\omega)$ is Poincar\'e dual to $2B+2F$. 
\end{itemize}
It is called reduced if and $\frac{\omega(B)}{\omega(F)}\in \Delta_{-1}$. Note that a reduced framing always exists for any symplectic rational surface by \cite{KK17}. To carry out the canonical construction of ATF presentations for LCY, we will rely on the following definition.

\begin{definition}
    A {\bf framing} $\mathfrak{f}$ on the set $\mathbb{LCY}$ is a choice of representative $(X,\omega,D)$ together with a reduced framing on $(X,\omega)$ for each isomorphism class $\mathbb{X}\in\mathbb{LCY}$. 
\end{definition}

\subsection{Symplectic reduced model $\XX_{\text{red}}$}\label{section:reducedmodel}

  By the property of $\Delta_n$, there is a well-defined map $$p:\mathbb{LCY}\rightarrow \bigsqcup_{n\geq -1}\Delta_n,$$ obtained by taking the symplectic class. This map is known to have finite fibers and \cite{Enumerate} provides a general counting formula when the fiber is over a restrictive region within $\Delta_n$. A nice feature of considering a symplectic class in the restrictive region is that there is a one-to-one correspondence between the elements in the fiber and the morphisms from specific objects in the category $\mathcal{LCY}$\footnote{These morphisms and objects were called the `blowup patterns' and `germs' in \cite{Enumerate}.}, which enables us to enumerate the morphisms instead. We now provide a more detailed discussion of these specific objects.
  
For any pair $(X,\omega,D)$, by successively blowing down toric or non-toric exceptional spheres of minimal symplectic area, \cite{limak} shows that there exists another pair $(X',\omega',D')$ with $b_2(X')\leq 2$ and $\text{Hom}([X',\omega',D'],[X,\omega,D])$ is not empty. The notion `symplectic minimal model' was used in \cite{limak} for the pair $(X',\omega',D')$ as a direct parallel of the one used in algebraic geometry (\cite[Theorem 2.4]{Friedman}). Due to the multiple choices of the exceptional sphere of minimal symplectic area at each step, the resulting $(X',\omega',D')$ (even the diffeomorphism type of $X'$) is generally not uniquely determined. For the purpose of enumeration, the refinement in \cite{Enumerate} resolves this non-uniqueness by introducing a framing, which specifies a distinguished exceptional sphere to blow down at each step. Given a framing $\mathfrak{f}$ on $\mathbb{LCY}$, \cite[Lemma 2.31]{Enumerate} shows that for any object $\XX=[X,\omega,D]\in\mathcal{LCY}_{\geq1 }$ with the reduced framing $\{H,E_1,\cdots,E_n\}\subseteq H_2(X;\ZZ)$, there exists a unique morphism from a specific $\XX'=[X',\omega',D']$ to $\XX$, determined by successively blowing down exceptional spheres in class $E_n,E_{n-1},\cdots$,  until 
\begin{itemize}
    \item either the underlying manifold becomes $\CC\PP^2\#\overline{\CC\PP}^2$;
    \item or the underlying manifold becomes $\CC\PP^2\#l\overline{\CC\PP}^2$ with $l\geq 2$, and the number of divisor components becomes $2$ with one component representing $E_l$.
\end{itemize} 
To distinguish it from the minimal model used in \cite{limak}, it will be called the {\bf symplectic reduced model} of $\XX$ and denoted by $\XX_{\text{red}}$. There is also a unique morphism $\Gamma\in \text{Hom}(\XX_{\text{red}},\XX)$ indicating the blowdown procedure.

The possible homology sequence of $\XX_{\text{red}}$ falls into one of the following cases: 
\begin{enumerate}[label=(\roman*)]
	\item $(2H,H-E_1)$;
    \item $((a+1)H-aE_1,(-a+2)H+(a-1)E_1)_{a\in\ZZ_+}$;
	\item $(aH+(-a+1)E_1,H-E_1,(-a+2)H+(a-1)E_1)_{a\in\ZZ_+}$;
	\item $(aH+(-a+1)E_1,H-E_1,(-a+1)H+aE_1,H-E_1)_{a\in\ZZ_+}$;
    \item $(3H-E_1-\cdots-E_{l-1}-2E_l,E_l)$.
\end{enumerate}
In the first four cases, $X'=\CC\PP^2\#\overline{\CC\PP}^2$; in the final case, $X'=\CC\PP^2\#l\overline{\CC\PP}^2$ with some $2\leq l\leq n$. Therefore, unlike minimal models, $X'$ will not be $S^2\times S^2$. Also, note that the divisors in type (\rn{2}), (\rn{3}), (\rn{4}) arise as unions of the fiber and certain sections of the non-trivial $S^2$-bundle. See \cite[Theorem 2.14]{Enumerate} for the dual graphs of these configurations.

Let us emphasize that the symplectic reduced model depends on the framing $\mathfrak{f}$. For symplectic classes lying on the boundary of the $c_1$-positive polytope, it is possible for two distinct morphisms, originating from different objects, to terminate at the same target object as shown in the example below.
Recall that our convention is to use the notation $H_{i_1\cdots i_k}$ to denote the homology class $H-E_{i_1}-\cdots-E_{i_k}$.

\begin{example}\label{exa:framing}
    Consider $X=\CC\PP^2\#3\overline{\CC\PP}^2$ equipped with a monotone symplectic form $\omega_{\text{mon}}$, framing $\{H,E_1,E_2,E_3\}$ and $D=C_H\cup C_{H_3}\cup C_{H_{12}}$ which contains three components in classes $H,H_3,H_{12}$. By a symplectomorphism generated by the Dehn twist along the Lagrangian sphere in class $H_{123}$, there will be another framing $\{h,e_1,e_2,e_3\}:=\{2H-E_1-E_2-E_3,H_{23},H_{13},H_{12}\}$. Thus, following the blowdown recipe to obtain symplectic reduced models, these two framings will give rise to
\[\XX^{(1)}_{\text{red}}=[\CC\PP^2\#\overline{\CC\PP}^2,\omega_{1},C_H\cup C_H\cup C_{H_1}],\,\, \XX^{(2)}_{\text{red}}=[\CC\PP^2\#\overline{\CC\PP}^2,\omega_{2},C_{2h-e_1}\cup C_h]\]
which are of type (\rn{3}) and (\rn{2}) respectively.\footnote{Indeed, $\omega_1$ and $\omega_2$ are actually symplectomorphic.} A description for the morphisms is given by
    \[(H,H,H_1)\xrightarrow[]{\text{nontoric}}(H,H,H_{12})\xrightarrow[]{\text{nontoric}}(H,H_3,H_{12})\in \text{Hom}(\XX_{\text{red}}^{(1)},[X,\omega_{\text{mon}},D])\]
    \[(2h-e_1,h)\xrightarrow[]{\text{nontoric}}(2h-e_1-e_2,h)\xrightarrow[]{\text{toric}}(2h-e_1-e_2-e_3,h_3,e_3)\in \text{Hom}(\XX_{\text{red}}^{(2)},[X,\omega_{\text{mon}},D]).\]
\end{example}

\subsection{$\varepsilon$-replacement $\XX_{\varepsilon}$}\label{subsection:replace}



From now on, we always fix a framing $\mathfrak{f}$ on $\mathbb{LCY}$. Given some $\XX=[X,\omega,D]\in\mathcal{LCY}_{\geq 2}$, suppose we have the symplectic reduced model $\XX_{\text{red}}$ together with a unique morphism $\Gamma\in \text{Hom}(\XX_{\text{red}},\XX)$ of the form
\[\XX_{\text{red}}\xrightarrow{f_1}\cdots\xrightarrow{f_{n-1}} \XX, \,\,\,\, \text{if }\XX_{\text{red}}\text{ is of type (\rn{1}), (\rn{2}), (\rn{3}), (\rn{4})};\]
\[\XX_{\text{red}}\xrightarrow{f_l}\cdots\xrightarrow{f_{n-1}} \XX, \,\,\,\, \text{if }\XX_{\text{red}}\text{ is of type (\rn{5})},\]where each $f_k$ denotes either a toric or non-toric blowup producing an exceptional sphere in class $E_{k+1}$. Let us introduce the following index set encoding the position of each toric blowup
\[\mathcal{D}_i:=\{2\leq k\leq n\,|\,[D_i]\cdot E_k=1, f_{k-1}\text{ is a toric blowup}\},\,\,\,\, \text{if }\XX_{\text{red}}\text{ is of type (\rn{1}), (\rn{2}), (\rn{3}), (\rn{4})};\]
\[\mathcal{D}_i:=\{l+1\leq k\leq n\,|\,[D_i]\cdot E_k=1, f_{k-1}\text{ is a toric blowup}\},\,\,\,\, \text{if }\XX_{\text{red}}\text{ is of type (\rn{5})}\]
for any component $D_i\subseteq D$. Let $\varepsilon$ be a positive real number which is sufficiently small in the sense that $\varepsilon\ll$ min$\{\omega(D_i)\,|\,D_i\text{ is a component of } D\}$. The aim of this section is to provide a recipe for constructing a canonical assignment $\XX_{\varepsilon}\in \mathbb{LCY}$, which will be called the {\bf $\varepsilon$-replacement}, based on the data of $\XX_{\text{red}}$ and $\Gamma$ (indeed, just $f_1$ and certain $\mathcal{D}_i$'s). 
It adds more components to the divisor by performing toric blowups, making it possible 
to blow down non-toric exceptional spheres to obtain a toric pair. A simple example is given by the pair with homology sequence
$(2H - E_{1},\, H - E_{2})$, where it is impossible to obtain a toric pair 
without first carrying out a toric blowup. As mentioned in the introduction, this $\XX_{\varepsilon}$ is the symplectic analogue of intermediate pair $(\widetilde{Y},\widetilde{D})$ in the construction of the holomorphic toric model (\cite{GHK}) $(\overline{Y},\overline{D})$ for the anticanonical pair $(Y,D)$:
\begin{align}\label{equation:holtoricmodel}
    (\overline{Y},\overline{D})\xleftarrow{\text{non-toric}}(\widetilde{Y},\widetilde{D})\xrightarrow{\text{toric}}(Y,D).
\end{align}

The following five diagrams \ref{fig:curve1},  \ref{fig:curve2}, \ref{fig:curve3}, \ref{fig:curve4} and \ref{fig:curve5} depict divisor configurations corresponding to the five types of symplectic reduced models discussed in Section \ref{section:reducedmodel}, and illustrate the possible blowup patterns $\Gamma$ from $\XX_{\text{red}}$ to $\XX$. Note that we use `$\cdots$' in the configurations to denote the potential components between two components drawn as curves in the figures. Let us now define $\XX_{\varepsilon}$ for each case individually.

 When $\XX_{\text{red}}$ is of type (\lowerromannumeral{1}), let $D_1,D_2\subseteq D$ be the proper transforms of the $H_1,2H$-components of $\XX_{\text{red}}$. There are two subcases determined by the first blowup $f_1$ in $\Gamma$.  First, suppose $f_1$ is the non-toric blowup at the $2H$-component. We will refer to this case as type (\rn{1})-$a$.
\begin{itemize}
    \item If $\#\mathcal{D}_1\cap\mathcal{D}_2\geq2$, then $[D_1]\cdot[D_2]=0$ and we define $\XX_{\varepsilon}:=\XX$.
     \item If $\#\mathcal{D}_1\cap\mathcal{D}_2=1$, then $[D_1]\cdot[D_2]=1$ and we define $\XX_{\varepsilon}$ to be the toric blowup centered at the node $D_1\cap D_2$ of size $\varepsilon$.
      \item If $\#\mathcal{D}_1\cap\mathcal{D}_2=0$, then $[D_1]\cdot[D_2]=2$ and we define $\XX_{\varepsilon}$ to be the two-fold toric blowups centered at two nodes $D_1\cap D_2$ both of size $\varepsilon$.
\end{itemize}
Otherwise, suppose $f_1$ is not the non-toric blowup at the $2H$-component. We will refer to this case as type (\rn{1})-$b$.
\begin{itemize}
    \item If $\mathcal{D}_2\not\subseteq \mathcal{D}_1$, define $\XX_{\varepsilon}:=\XX$.
     \item If $\mathcal{D}_2\subseteq \mathcal{D}_1$ but $\mathcal{D}_2$ is non-empty, then the proper transform of the toric exceptional sphere produced by $f_{\min \mathcal{D}_2-1}$ will intersect $D_2$ at a node. Define $\XX_{\varepsilon}$ to be the toric blowup centered at that node of size $\varepsilon$.
      \item If $\mathcal{D}_2$ is empty, then so is $\mathcal{D}_1$ and $D_1\cap D_2$ must have two nodes. Let us first perform a toric blowup at one node of size $\varepsilon$, which produces a toric exceptional component $C_{E_{n+1}}$. Then we perform another toric blowup at the node between $C_{E_{n+1}}$ and the proper transform of $D_2$ of size $\frac{\varepsilon}{2}$ to define $\XX_{\varepsilon}$. It is a consequence of the Torelli Theorem \ref{prop:Torelli} that the class $\XX_{\varepsilon}$ does not depend on choice of the node in $D_1\cap D_2$.
\end{itemize}

    \begin{figure*}[]\label{fig:curve1}
		\includegraphics*[width=\linewidth]{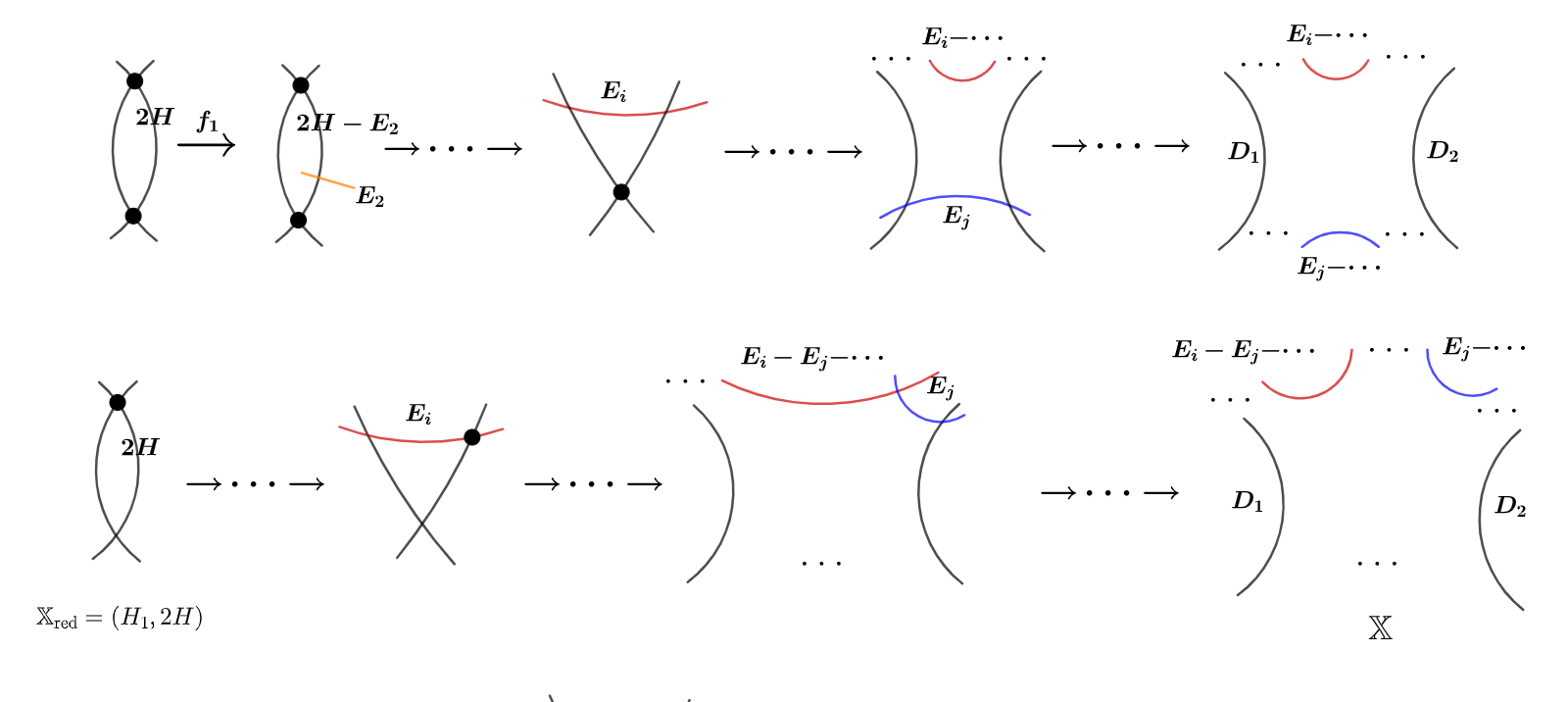}
  \caption{$\XX_{\text{red}}$ is of type (\rn{1}).}\label{fig:curve1}
	\end{figure*}

\begin{rmk}\label{rmk:f1}
    The reason we examine $f_1$ and divide type (\rn{1}) into two subcases is that we require that the symplectic reduced model can only be $\CC\PP^2\#\overline{\CC\PP}^2$. Many LCYs can be realized as blowups of minimal models in both $\CC\PP^2\#\overline{\CC\PP}^2$ and $S^2\times S^2$. For the construction of symplectic toric model in the next section, it is more natural to view some of these as blowups of LCYs in $S^2\times S^2$. However, to avoid introducing additional types beyond the current classification (which we keep to five), we opt not to treat $S^2\times S^2$ explicitly and instead introduce two subcases of type (i), distinguished by the first blow-up $f_1$.
\end{rmk}
    
When $\XX_{\text{red}}$ is of type (\lowerromannumeral{2}), let $D_1,D_2\subseteq D$ be the proper transforms of two components of $\XX_{\text{red}}$. 
\begin{itemize}
   \item If $\#\mathcal{D}_1\cap\mathcal{D}_2\geq2$, then $[D_1]\cdot[D_2]=0$ and we define $\XX_{\varepsilon}:=\XX$.
     \item If $\#\mathcal{D}_1\cap\mathcal{D}_2=1$, then $[D_1]\cdot[D_2]=1$ and we define $\XX_{\varepsilon}$ to be the toric blowup centered at the node $D_1\cap D_2$ of size $\varepsilon$.
      \item If $\#\mathcal{D}_1\cap\mathcal{D}_2=0$, then $[D_1]\cdot[D_2]=2$ and we define $\XX_{\varepsilon}$ to be the two-fold toric blowups centered at two nodes $D_1\cap D_2$ both of size $\varepsilon$.
\end{itemize}

      \begin{figure*}[]\label{fig:curve2}
		\includegraphics*[width=\linewidth]{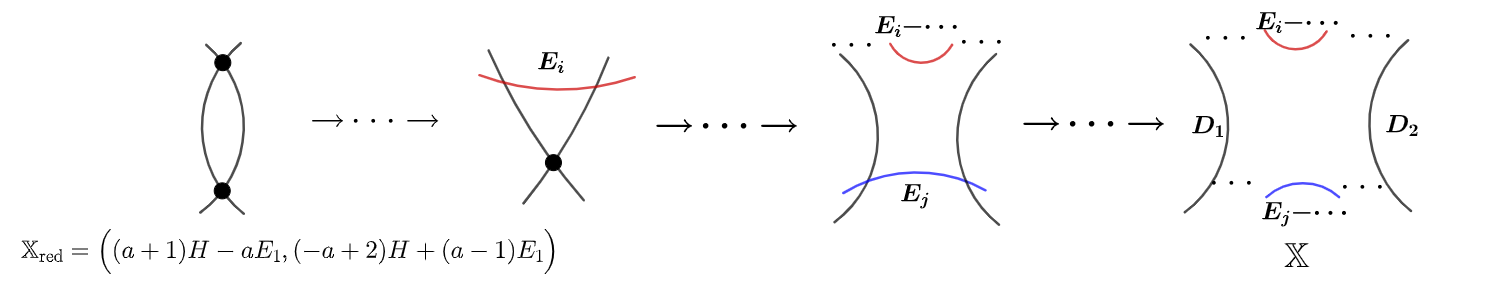}
  \caption{$\XX_{\text{red}}$ is of type (\rn{2}). }\label{fig:curve2}
	\end{figure*}
    
When $\XX_{\text{red}}$ is of type (\lowerromannumeral{3}), let $D_1,D_2\subseteq D$ be the proper transforms of two sections in classes $aH+(-a+1)E_1,(-a+2)H+(a-1)E_1$ of $\XX_{\text{red}}$. 
\begin{itemize}
   \item If $\#\mathcal{D}_1\cap\mathcal{D}_2\geq 1$, then $[D_1]\cdot[D_2]=0$ and we define $\XX_{\varepsilon}:=\XX$.
    \item If $\#\mathcal{D}_1\cap\mathcal{D}_2=0$, then $[D_1]\cdot [D_2]=1$ and we define $\XX_{\varepsilon}$ to be the toric blowup centered at the node $D_1\cap D_2$ of size $\varepsilon$.  
\end{itemize}  

   \begin{figure*}[]\label{fig:curve3}
		\includegraphics*[width=\linewidth]{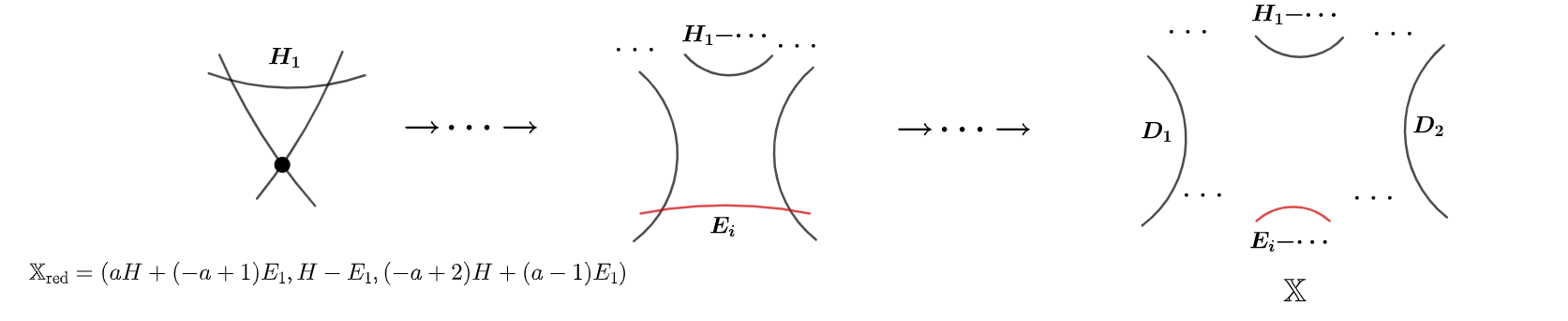}
  \caption{$\XX_{\text{red}}$ is of type (\rn{3}). }\label{fig:curve3}
	\end{figure*}

When $\XX_{\text{red}}$ is of type (\lowerromannumeral{4}), we always define $\XX_{\varepsilon}:=\XX$.

  \begin{figure*}[]\label{fig:curve4}
		\includegraphics*[width=\linewidth]{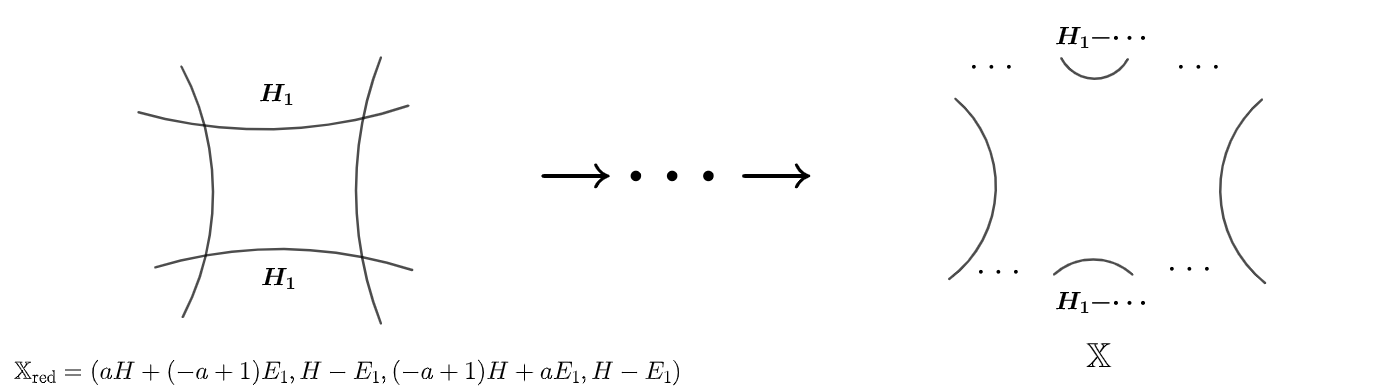}
  \caption{$\XX_{\text{red}}$ is of type (\rn{4}). }\label{fig:curve4}
	\end{figure*}

When $\XX_{\text{red}}$ is of type (\lowerromannumeral{5}), let $D_1,D_2\subseteq D$ be the proper transforms of two components of $\XX_{\text{red}}$. 
\begin{itemize}
   \item If $\#\mathcal{D}_1\cap\mathcal{D}_2\geq2$, then $[D_1]\cdot[D_2]=0$ and we define $\XX_{\varepsilon}:=\XX$.
     \item If $\#\mathcal{D}_1\cap\mathcal{D}_2=1$, then $[D_1]\cdot[D_2]=1$ and we define $\XX_{\varepsilon}$ to be the toric blowup centered at the node $D_1\cap D_2$ of size $\varepsilon$.
      \item If $\#\mathcal{D}_1\cap\mathcal{D}_2=0$, then $[D_1]\cdot[D_2]=2$ and we define $\XX_{\varepsilon}$ to be the two-fold toric blowups centered at two nodes $D_1\cap D_2$ both of size $\varepsilon$.
\end{itemize}
  \begin{figure*}[]\label{fig:curve5}
		\includegraphics*[width=\linewidth]{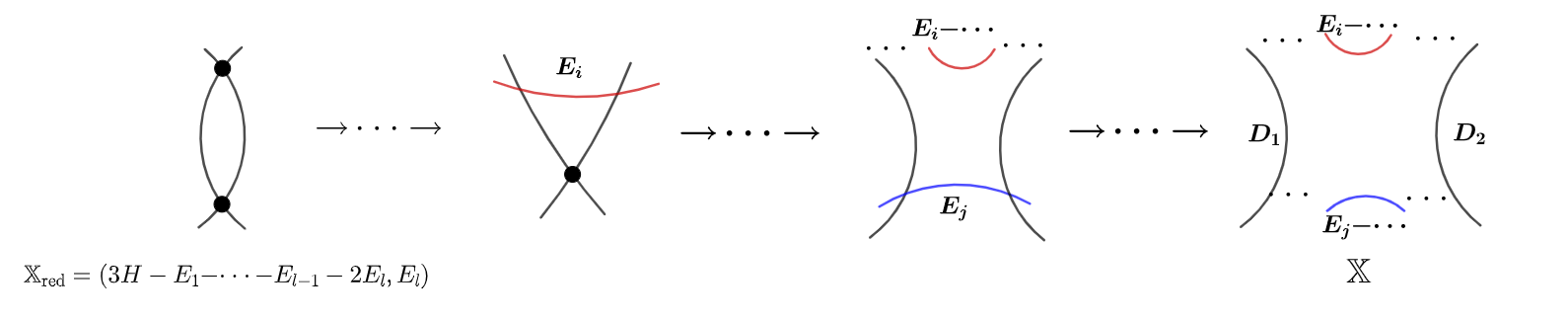}
  \caption{$\XX_{\text{red}}$ is of type (\rn{5}). }\label{fig:curve5}
	\end{figure*}
 
 Therefore, there will always be some morphism $\Gamma_\varepsilon\in \text{Hom}(\XX,\XX_\varepsilon)$ which either records the small blowup when $\XX_{\varepsilon}\neq \XX$,  or is the identity when $\XX_\varepsilon=\XX$. Note that this $\Gamma_\varepsilon$ is not unique since we do not specify the order for two blowups for type (\rn{1})-$a$, (\rn{2}) and  (\rn{5}) when $\#\mathcal{D}_1\cap\mathcal{D}_2\geq 2$. 
 
 Now, let us further introduce the following indices $i$ and $j$, defined according to the procedure below.
 \begin{itemize}
     \item For type (\rn{1})-$a$, (\rn{2}) and (\rn{5}), let $i:=\min \{\mathcal{D}_1\cap\mathcal{D}_2,n+1\}$ and  $j:=\min \{\mathcal{D}_1\cap\mathcal{D}_2\setminus\{i\},n+2-\#\mathcal{D}_1\cap\mathcal{D}_2\}$.
     \item For type (\rn{3}), let $i:=\min \{\mathcal{D}_1\cap\mathcal{D}_2,n+1\}$.
      \item For type (\rn{1})-$b$, if $\mathcal{D}_2\not\subseteq\mathcal{D}_1$,
      let $j:=\min \{\mathcal{D}_2\setminus\mathcal{D}_1\}$ and $i$ be one of the two elements in $\mathcal{D}_1\cap\mathcal{D}_2$ such that the proper transform in $D$ of the toric exceptional sphere produced by $f_{i-1}$ has non-trivial pairing with $E_j$; if $\mathcal{D}_2\subseteq\mathcal{D}_1$ but $\mathcal{D}_2$ is non-empty, let $i:=\min\{\mathcal{D}_2\}$ and $j:=n+1$; if $\mathcal{D}_2$ is empty, let $i:=n+1$ and $j:=n+2$.
 \end{itemize}
 Intuitively, these two indices record the first one or two toric blowups that ensure the divisor 
contains the proper transforms of four components in the toric Hirzebruch pair $\XX_{\text{tor}}'$ which will be defined in the next section. Now, we will use these two indices $i,j$ to define two additional quantities $\dd_i,\dd_j$ based on the following recipe. Note that in each case, certain nodes have been marked as dots in the diagrams above. The blowup patterns $\Gamma$ shown in the diagrams can be deceptive, as they all include toric blowups at those marked nodes which produce the red and blue curves shown. In general, $\Gamma$ may or may not contain these toric blowups.

 If $\Gamma$ does include them, then $\XX=\XX_{\varepsilon}$, and 
 we will record the symplectic areas of the toric exceptional classes $E_i,E_j$ as $\dd_i,\dd_j$ respectively. Otherwise, there will be one or two marked nodes in the configuration diagrams which do not admit toric blowups (meaning the red or blue curve actually does not appear). In such situations, $\XX_{\varepsilon}\neq \XX$ and the indices $i$ or $j$ may exceed $n$. In the special case of type (\rn{1})-$b$, if $i=n+1$ and $j=n+2$, we will let $\dd_i:=\varepsilon$ and $\dd_j:=\frac{\varepsilon}{2}$; while in all other cases, we will let $\dd_i$ and $\dd_j$ be $\varepsilon$ whenever the index $i$ or $j$ is greater than $n$. The quantities $\dd_i,\dd_j$ will be referred to as the {\bf distinguished blowup sizes} and used in the definition of the symplectic toric model in the next section.

 
  \begin{example}\label{example:ex}
  Let us work through some examples to clarify the assignment of $\XX_{\varepsilon}$ and the distinguished blowup sizes $\dd_i,\dd_j$ defined in this section. The first column of the following table indicates the homology type for $\XX$, where we adopt the convention $H_{i_1\cdots i_k}=H-E_{i_1}-\cdots-E_{i_k}$. Their reduced models and types are presented in the second and third columns respectively. Although the symplectic form is not explicitly specified, one can just simply assume that $\omega$ in $\XX$ satisfies $\omega(E_{i+1})\ll\omega(E_i)$ for all $i$.
     
     \begin{center}
\begin{tabular}{ |c|c|c| } 
 \hline
 $\XX$ & $\XX_{\text{red}}$ & \text{Type} \\ 
 \hline
 $(H_{134},E_3,H+H_{23})$ & $(H_1,2H)$ & (\rn{1})-$a$\\ 
 \hline
$(H_{1234},E_3,H+H_{34},E_4)$ & $(H_1,2H)$ & (\rn{1})-$b$\\ 
 \hline
$(H_{123},2H-E_4)$ & $(H_1,2H)$ & (\rn{1})-$b$ \\
\hline
$(8H_1+H_{2},-7H_1+H_3)$ & $(8H_1+H,-7H_1+H)$ & (\rn{2})  \\
\hline
$(H_{1234},E_3,4H-3E_1-E_3,-2H+3E_1)$& $(H_1,4H-3E_1,-2H+3E_1)$ & (\rn{3})\\
\hline
$(H_{14},2H-E_1-E_2,E_2,H_{123},-H+2E_1)$& $(H_1,2H-E_1,H_1,-H+2E_1)$& (\rn{4})\\
\hline
$(2H+H_{1234567}-E_6,E_7-E_8,E_6-E_7)$& $(2H+H_{12345}-2E_6,E_6)$& (\rn{5}) \\
\hline
\end{tabular}
\end{center}

The table below lists the corresponding $\varepsilon$-replacement, index pairs $(i,j)$ and distinguished blowup sizes $\dd_i,\dd_j$. Note that $j,\dd_j$ are not defined for type (\rn{3}), and $i,j,\dd_i,\dd_j$ are not defined for type (\rn{4}).

  \begin{center}
\begin{tabular}{ |c|c|c| } 
 \hline
  $\XX_{\varepsilon}$ & $(i,j)$& $(\dd_i,\dd_j)$ \\ 
 \hline
 $(H_{1345},E_3,H+H_{235},E_5)$& $(3,5)$ &$(\omega(E_3),\varepsilon)$\\ 
 \hline
$(H_{1234},E_3-E_5,E_5,H+H_{345},E_4)$ &  $(3,5)$&$(\omega(E_3),\varepsilon)$\\ 
 \hline
  $(H_{1235},E_5-E_6,E_6,H+H_{456})$ & $(5,6)$ & $(\varepsilon,\frac{\varepsilon}{2})$\\
\hline
$(8H_1+H_{245},E_4,-7H_1+H_{345},E_5)$&$(4,5)$& $(\varepsilon,\varepsilon)$ \\
\hline
$(H_{1234},E_3,4H-3E_1-E_3,-2H+3E_1)$ &$i=3$, $j$ undefined & $\dd_i=\omega(E_3)$, $\dd_j$ undefined\\
\hline
$(H_{14},2H-E_1-E_2,E_2,H_{123},-H+2E_1)$& undefined & undefined\\
\hline
$(2H+H_{12345679}-E_6,E_7-E_8,E_6-E_7-E_9,E_9)$& $(7,9)$& $(\omega(E_7),\varepsilon)$ \\
\hline
\end{tabular}
\end{center}
    
  \end{example}   

\subsection{Symplectic toric model $\XX_{\text{tor}}$}\label{section:toricmodel}
Given $\XX=[X,\omega,D]\in \mathcal{LCY}_{\geq2}$, consider its symplectic reduced model $\XX_{\text{red}}$ and $\varepsilon$-replacement $\XX_{\varepsilon}$ introduced in the previous sections. The composition of the canonical $\Gamma\in\text{Hom}(\XX_{\text{red}},\XX)$ with some non-canonical $\Gamma_\varepsilon\in\text{Hom}(\XX,\XX_\varepsilon)$ yields a blowup pattern $\Gamma_\varepsilon\circ\Gamma\in\text{Hom}(\XX_{\text{red}},\XX_\varepsilon)$:
\[\XX_{\text{red}}\xrightarrow{f_1}\cdots\xrightarrow{f_m} \XX_\varepsilon. \]

The purpose of this section is to modify the above blowup patterns to obtain a new morphism in $\text{Hom}(\XX_{\text{tor}}',\XX_\varepsilon)$:
\[\XX_{\text{tor}}'\xrightarrow{t_1}\cdots\xrightarrow{t_s}\XX_{\text{tor}}\xrightarrow{n_1}\cdots\xrightarrow{n_r} \XX_\varepsilon,\]
where 
\begin{itemize}
    \item $n_1,\cdots,n_r$ are non-toric blowups;
    \item $t_1,\cdots,t_s$ are toric blowups;
    \item $\XX_{\text{tor}}$ is a canonically assigned toric pair, which is the analogue of $(\widetilde{Y},\widetilde{D})$ in (\ref{equation:holtoricmodel}) and will be called the {\bf symplectic toric model};
    \item $\XX_{\text{tor}}'$ is another canonically assigned toric pair which is the boundary divisor of a toric Hirzebruch surface.
\end{itemize}

Let us begin with the construction of the non-toric blowups $n_1,\cdots,n_r$ and the symplectic toric model $\XX_{\text{tor}}$. The following lemma enables us to artificially select certain non-toric exceptional classes to perform blowdowns.


\begin{lemma}\label{lemma:exceptional}
    Given a set $\mathcal{E}=\{S_1,\cdots,S_k\}$ of pairwise orthogonal non-toric exceptional classes, one can choose their embedded symplectic representatives $C_{S_1},\cdots,C_{S_k}$ which are pairwise disjoint, with each intersecting exactly one component of $D$ transversely.
\end{lemma}
\begin{proof}
    As a direct consequence of \cite[Theorem 1.2.7]{MO} (or \cite[Lemma 2.4]{limak}), one can obtain an $\omega$-tamed almost complex structure $J$ such that each component in $D$ and every $S_i$ are represented by $J$-holomorphic spheres. The result then follows from positivity of intersection.
\end{proof}


Consider the partition $\{1,\cdots,m\}=\mathcal{N}\cup\mathcal{T}$ where
\[\mathcal{N}:=\{k\in[1,m]\,|\,f_k\text{ is a non-toric blowup}\},\] 
\[\mathcal{T}:=\{k\in[1,m]\,|\,f_k\text{ is a toric blowup}\}\] are the index sets for $f_*$ being non-toric/toric blowups. Now, for the $\varepsilon$-replacement $\XX_{\varepsilon}$, we choose the non-toric exceptional set $\mathcal{E}$ as follows:
\begin{enumerate}[label=(\roman*)]
    \item  
    $\begin{cases}
      \{E_1,H_{2i},H_{2j}\}\cup\{E_{k+1}\,|\, k\in\mathcal{N},k\geq 2\} , & f_1 \text{ is non-toric blowup on }2H \\
      \{E_1,H_{ij}\}\cup\{E_{k+1}\,|\, k\in\mathcal{N}\}, & \text{otherwise}
    \end{cases}$

    \item  
    $\begin{cases}
      \{H_{12},H_{1i},H_{1j}\}\cup\{E_{k+1}\,|\, k\in\mathcal{N},k\geq 2\} , & f_1 \text{ is non-toric blowup on }(a+1)H-aE_1 \\
      \{H_{1i},H_{1j}\}\cup\{E_{k+1}\,|\, k\in\mathcal{N}\}, & \text{otherwise}
    \end{cases}$
    
  \item  
    $\begin{cases}
      \{H_{12},H_{1i}\}\cup\{E_{k+1}\,|\, k\in\mathcal{N},k\geq 2\} , & f_1 \text{ is non-toric blowup on }aH-(a-1)E_1 \\
      \{H_{1i}\}\cup\{E_{k+1}\,|\, k\in\mathcal{N}\}, & \text{otherwise}
    \end{cases}$
    
      \item  
        $\begin{cases}
      \{H_{12}\}\cup\{E_{k+1}\,|\, k\in\mathcal{N},k\geq 2\} , & f_1 \text{ is non-toric blowup on }aH-(a-1)E_1 \\
      \{E_{k+1}\,|\, k\in\mathcal{N}\}, & \text{otherwise}
    \end{cases}$
    
      \item  
    $\{E_3,E_4,\cdots,E_{l-1},H_{1l},H_{2l},H_{li},H_{lj}\}\cup\{E_{k+l}\,|\, k\in\mathcal{N}\}$
\end{enumerate}

Note that the choice of $\mathcal{E}$ depends on $f_1$, for the same reason discussed in Remark \ref{rmk:f1}. One can easily verify that all the cases listed above satisfy the condition of Lemma \ref{lemma:exceptional}. Therefore, we can perform non-toric blowdowns for all such exceptional spheres. By Torelli Theorem \ref{prop:Torelli} (and the discussion following it), the isomorphism class of the pair after blowdowns is uniquely determined once the exceptional classes are specified. We thus obtain a canonically assigned $\XX_{\text{tor}}$. Moreover, $\XX_{\text{tor}}$ is indeed a toric pair,   as confirmed by a straightforward calculation of the change in $b_2$ and the number of components. The corresponding non-toric blowups will be denoted by $n_1,\cdots,n_r$, giving rise to some morphism in $\Gamma_t\in\text{Hom}(\XX_{\text{tor}},\XX_{\varepsilon})$. However, since the order of these non-toric blowups is not specified, $\Gamma_t$ is actually not canonically determined. 


We now carefully choose $\XX_{\text{tor}}'$, the boundary divisor of some symplectic toric Hirzebruch surface,  with a sequence of toric blowups $t_1,\cdots,t_s$ as the toric minimal reduction for $\XX_{\text{tor}}$. This process will also be canonical. Recall that all symplectic toric Hirzebruch surfaces modulo equivariant symplectomorphisms are classified by their moment polygons modulo $\text{AGL}(2,\ZZ)$-action. These polygons are all trapezoids whose edges have affine lengths $x,y,x,z$ (in a cyclic order), satisfying the relation $y-z=kx$ for some $k\in\ZZ$ (will be simply called a {\bf $k$-trapezoid} later). Therefore, the moduli can be described by \[\{(k,x,y,z)\,|\,y-z=kx\}\subseteq \ZZ\times \RR_+^3\] modulo the equivalence relations \[(k,x,y,z)\sim (-k,x,z,y)\,\,\,\text{ and }\,\,\,(0,x,y,y)\sim(0,y,x,x).\] 

By the correspondence between toric actions and toric LCY pairs (\cite[Theorem 1.5]{Enumerate}), the pair $\XX_{\text{tor}}'$ can be defined by specifying a $k$-trapezoid,  represented by the quadruple $(k,x,y,z)$ as follows:
\begin{enumerate}[label=(\roman*)]
\item $\begin{cases}
      (1,1-\dd_2,2-\dd_2-\dd_i-\dd_j,1-\dd_i-\dd_j) , & f_1 \text{ is non-toric blowup on }2H; \\
     (2,1-\dd_i,2-\dd_i-\dd_j,\dd_i-\dd_j), & \text{otherwise}
    \end{cases}$
    


 \item $\begin{cases}
      (2a-2,1-\dd_1,1+a-a\dd_1-\dd_2-\dd_i-\dd_j,3-a+(a-2)\dd_1-\dd_2-\dd_i-\dd_j) , \\ f_1 \text{ is non-toric blowup on }(a+1)H-aE_1; \\
     (2a-1,1-\dd_1,a+1-a\dd_1-\dd_i-\dd_j,-a+2+(a-1)\dd_1-\dd_i-\dd_j),\text{otherwise}
    \end{cases}$
    


     \item $\begin{cases}
      (2a-3,1-\dd_1,a-(a-1)\dd_1-\dd_2-\dd_i,3-a+(a-2)\dd_1-\dd_2-\dd_i) , \\ f_1 \text{ is either non-toric blowup on }aH+(-a+1)E_1 \text{ or}\\
      \text{ toric blowup at the node between } aH+(-a+1)E_1 \text{ and } H-E_1; \\
     (2a-2,1-\dd_1,a-(a-1)\dd_1-\dd_i,-a+2+(a-1)\dd_1-\dd_i),\text{ otherwise}
    \end{cases}$
    


 \item $\begin{cases}
      (2a-2,1-\dd_1,a-(a-1)\dd_1-\dd_2,2-a+(a-1)\dd_1-\dd_2) ,\\ f_1 \text{ is either non-toric blowup on }aH+(-a+1)E_1 \text{ or}\\
      \text{ toric blowup at the node between } aH+(-a+1)E_1 \text{ and } H-E_1; \\
     (2a-1,1-\dd_1,a-(a-1)\dd_1,1-a+a\dd_1),\text{ otherwise}
    \end{cases}$



    \item $(1,1-\dd_l,3-\dd_1-\dd_2-2\dd_l-\dd_i-\dd_j,2-\dd_1-\dd_2-\dd_l-\dd_i-\dd_j)$
    
\end{enumerate}

The sizes $x,y,z$ above, which may appear mysterious at first glance, actually correspond to the `initial' areas of the four (colored black, red and blue in Figures \ref{fig:curve1}, \ref{fig:curve2}, \ref{fig:curve3}, \ref{fig:curve4} and \ref{fig:curve5}) curves in the pattern $\Gamma_\varepsilon\circ\Gamma$, plus the areas of the non-toric exceptional classes  in the chosen set $\mathcal{E}$ that have non-trivial pairing with these four curve classes. Here, `initial' refers to the area of each curve prior to any blowup in $\Gamma_\varepsilon\circ\Gamma$ applied to it. For the black curves, this refers to the areas of the corresponding curves in $\XX_{\text{red}}$; for red or blue ones, it simply refers to $\dd_i$ or $\dd_j$.  See Figures \ref{fig:toric1}, \ref{fig:toric2}, \ref{fig:toric3}, \ref{fig:toric4} and  \ref{fig:toric5}, where the green dashed curves denotes the non-toric exceptional spheres corresponding to the set  $\mathcal{E}\setminus\{E_{k+1}\,|\,k\in \mathcal{N}\}$, and the purple dashed curves denotes the remaining non-toric exceptional spheres we choose to blow down in order to obtain $\XX_{\text{tor}}$.

  \begin{figure}[]
     \centering
		\includegraphics[width=8cm]{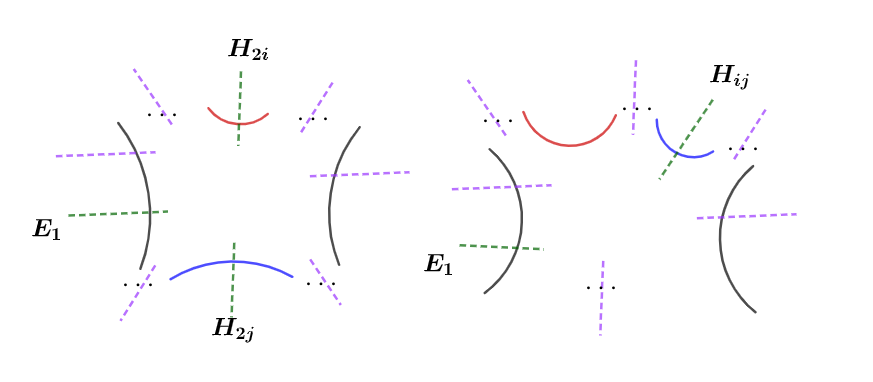}
  \caption{Case (\rn{1}): These two diagrams correspond to whether $f_1$ is the non-toric blowup at the $2H$-component. In both diagrams, the left and right black curves denote the proper transforms of the $H_1$- and $2H$-components in the reduced model, respectively. }\label{fig:toric1}
	\end{figure}

      \begin{figure}[]
     \centering
		\includegraphics[width=7cm]{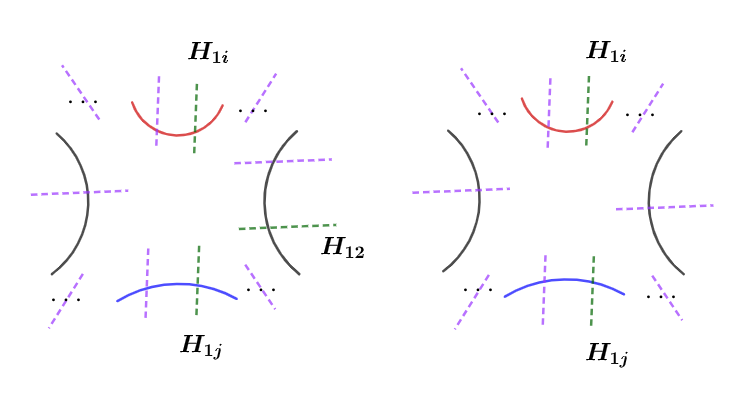}
  \caption{Case (\rn{2}): These two diagrams correspond to whether $f_1$ is the non-toric blowup at the $(a+1)H-aE_1$-component. In both diagrams, the left and right black curves denote the proper transforms of the $(a+1)H-aE_1$- and $(-a+2)H+(a-1)E_1$-components in the reduced model, respectively. }\label{fig:toric2}
	\end{figure}

      \begin{figure}[]
     \centering
		\includegraphics[width=7cm]{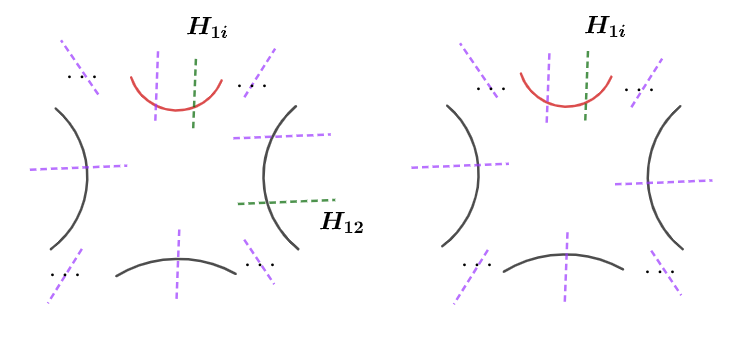}
  \caption{Case (\rn{3}): These two diagrams correspond to whether $f_1$ is the non-toric blowup at the $aH+(-a+1)E_1$-component. In both diagrams, the left, right and bottom black curves denote the proper transforms of the $aH+(-a+1)E_1$- ,$(-a+2)H+(a-1)E_1$- and $H_1$-components in the reduced model, respectively. }\label{fig:toric3}
	\end{figure}

      \begin{figure}[]
     \centering
		\includegraphics[width=7cm]{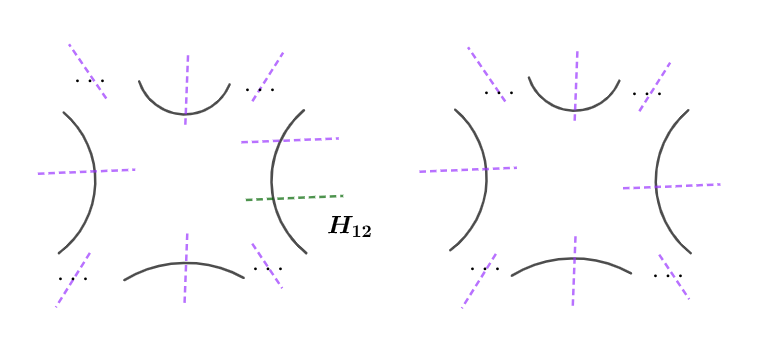}
  \caption{Case (\rn{4}): These two diagrams correspond to whether $f_1$ is the non-toric blowup at the $aH+(-a+1)E_1$-component. In both diagrams, the left, right, top and bottom black curves denote the proper transforms of the $aH+(-a+1)E_1$- ,$(-a+1)H+aE_1$-, $H_1$- and $H_1$-components in the reduced model, respectively. }\label{fig:toric4}
	\end{figure}

     \begin{figure}[]
     \centering
		\includegraphics[width=5cm]{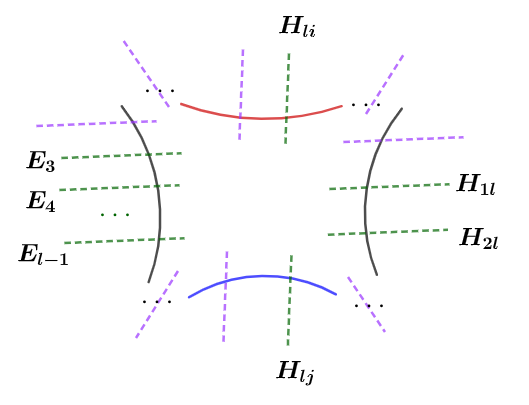}
  \caption{Case (\rn{5}): The left and right black curves denote the proper transforms of the $3H-E_1-\cdots-E_{l-1}-2E_l$- and $2E_l$-components in the reduced model, respectively.}\label{fig:toric5}
	\end{figure}

With this in mind, we can use the toric blowups $f_k$ for $k\in\mathcal{T}\setminus\{i,j\}$ appearing in $\Gamma_\varepsilon\circ\Gamma$ to build our desired $\Gamma_t\in\text{Hom}(\XX_{\text{tor}}',\XX_{\text{tor}})$ as follows. If $\mathcal{T}\setminus\{i,j\}$ is empty, then by our choice of $\mathcal{E}$ and the definition of the indices $i$ and $j$, $\XX_{\text{tor}}$ must have four components and we can simply take $\XX_{\text{tor}}':=\XX_{\text{tor}}$ and let $\Gamma_t$ be the identity. Otherwise, let $p:=\max\{\mathcal{T}\setminus\{i,j\}\}$. Note that $\mathcal{E}$ contains all the non-toric exceptional classes with non-trivial pairing with $E_p$. Therefore, $\XX_{\text{tor}}$ must contain a component in class $E_p$ on which we can perform toric blowdown. Then, we can continue the procedure of taking $q:=\max\{\mathcal{T}\setminus\{i,j,p\}\}$ and performing toric blowdown on the $E_q$-component. Repeat this procedure until all indices in $\mathcal{T}\setminus\{i,j,1\}$ have been exhausted. Then, for the same reason as in Remark \ref{rmk:f1}, we again have to discuss two cases determined by $f_1$:
\begin{itemize}
    \item when $f_1$ is the toric blowup at the node between $aH+(-a+1)E_1$ and $H-E_1$ in case (\rn{3}) or (\rn{4}), we perform toric blowdown on the $H_{12}$-component\footnote{The reason why we make such a choice will be clear in the proof of the main Theorem \ref{thm:main}, where we need to check the triangle packing problem is solvable.};
    \item otherwise, we perform toric blowdown on $E_2$-component if $1\in\mathcal{T}$.
\end{itemize} A direct computation of the area and self-intersection sequences\footnote{In fact, this is exactly how we choose $(k ,x,y,z)$ in the definition of $\XX_{\text{tor}}'$.} shows that, after performing all the toric blowdowns according to the procedure outlined above, the resulting pair is isomorphic to $\XX_{\text{tor}}'$, by the tautness of toric pairs (Lemma \ref{lem:torictaut}). Unlike the non-toric blowups $n_k$'s, the toric blowup sequence $\Gamma_t:\XX_{\text{tor}}'\xrightarrow{t_1}\cdots\xrightarrow{t_s}\XX_{\text{tor}}$ is a canonically assigned element in $\text{Hom}(\XX_{\text{tor}}',\XX_{\text{tor}})$ since each $t_k$ is completely determined by the blowup pattern $\Gamma_\varepsilon\circ\Gamma$ without any additional choice.

\hfill\break
We conclude this section with a summary of the main constructions introduced in Section \ref{section:lcy} for future reference. Given a framing $\mathfrak{f}$, any $\XX\in\mathcal{LCY}_{\geq2}$ has 
\begin{itemize}
    \item a canonical symplectic reduced model $\XX_{\text{red}}$ with a uniquely determined $\Gamma\in \text{Hom}(\XX_{\text{red}},\XX):  \XX_{\text{red}}\xrightarrow{}\cdots\rightarrow \XX $;
    \item  a canonical $\varepsilon$-replacement $\XX_\varepsilon$ with some non-uniquely determined $\Gamma_\varepsilon\in \text{Hom}(\XX,\XX_\varepsilon):\XX\rightarrow\cdots\rightarrow\XX_\varepsilon$; 
    \item a canonical symplectic toric model $\XX_{\text{tor}}$ with some non-uniquely determined $\Gamma_n\in \text{Hom}(\XX_{\text{tor}},\XX_{\varepsilon}): \mathbb{X}_{\text{tor}}\xrightarrow{}\cdots\rightarrow \XX_{\varepsilon} $;
    \item a canonical toric Hirzebruch surface pair $\XX_{\text{tor}}'$ with a uniquely determined $\Gamma_t\in \text{Hom}(\XX_{\text{tor}}',\XX_{\text{tor}}):\mathbb{X}_{\text{tor}}'\rightarrow\cdots\rightarrow \XX_{\text{tor}}$.
\end{itemize}

\section{Base diagram set $\mathbb{BD}$}\label{section:bd}

The aim of this section is to introduce a set $\mathbb{BD}$\footnote{$\mathbb{BD}$ could be interpreted as the initials for either `base diagram' or `bitten Delzant'.} consisting of certain diagrams that can be drawn in $\RR^2$. Each such diagram naturally gives rise to a nodal integral affine structure on the disk, which in turn serves as an almost toric base for a symplectic rational manifold. As mentioned in the introduction, a complete classification of nodal integral affine structures is out of reach. However, for our purpose of realizing all LCY pairs as the boundary divisor of some ATF, the elements in this set $\mathbb{BD}$ are already sufficient.

\subsection{Bitten Delzant polygons}\label{sec:bitten}
Let us start with some standard notions in integral affine geometry.
\begin{definition}\label{def:affine}
A nonzero vector $\overrightarrow{p}=(a,b)\in \ZZ^2$ is called {\bf primitive} if $gcd(a,b)=1$. Given a segment $AB\subseteq\RR^2$ with rational slope, there exists a unique primitive vector $\overrightarrow{p}$ such that $\overrightarrow{AB}=c\overrightarrow{p}$ with $c\in \RR_+$. We call such $c$ the {\bf affine length} of $AB$, denoted by $l(AB)$. If $O$ is a point in $\RR^2$, the {\bf affine distance} from $O$ to $AB$ is defined by choosing a smooth curve $\gamma:[0,1]\rightarrow \RR^2$ with $\gamma(0)=O,\gamma(1)\in AB$ and setting $$d(O,AB)=\int_{0}^1 | \overrightarrow{p}\times \gamma'(t)|\,dt$$
\end{definition}

It is straightforward to verify that the affine distance $d(O,AB)$ is independent of the choice of $\gamma$, since it can alternatively be defined as the magnitude of the inner product between $\overrightarrow{OA}$ and a primitive integral vector perpendicular to $AB$. In the context of almost toric base diagrams, it corresponds to the symplectic area of a visible symplectic surface lying over $\gamma$ (see \cite[Proposition 7.8]{symington}).

\begin{definition}
    For a Delzant polygon $P\subseteq \RR^2$, a {\bf regular Symington triangle} is an embedded triangle in $P$ such that one of its edge lies along the interior of an edge $AB$ of $P$ with affine length $d$, and the vertex opposite this edge has affine distance $d$ to $AB$.
\end{definition}
 Here, `regular' emphasizes that the size $d$ is strictly less than the affine length $l(AB)$ of the corresponding edge in $P$.

\begin{definition}
   A {\bf regular bitten Delzant polygon} is the datum of a Delzant polygon $P$ together with a finite collection of disjoint regular Symington triangles embedded in $P$.
\end{definition}

Recall that any Delzant polygon can be obtained from a $k$-trapezoid $P_0$ via a sequence of corner choppings. Given a regular bitten Delzant polygon, we may label all Symington triangles as well as the triangles in the corner chopping procedure. By ATF visible surface technique, one can read off the symplectic class from these labels. See Figure \ref{fig:delzant} for an example.

 \begin{figure}[h]
		\includegraphics*[width=\linewidth]{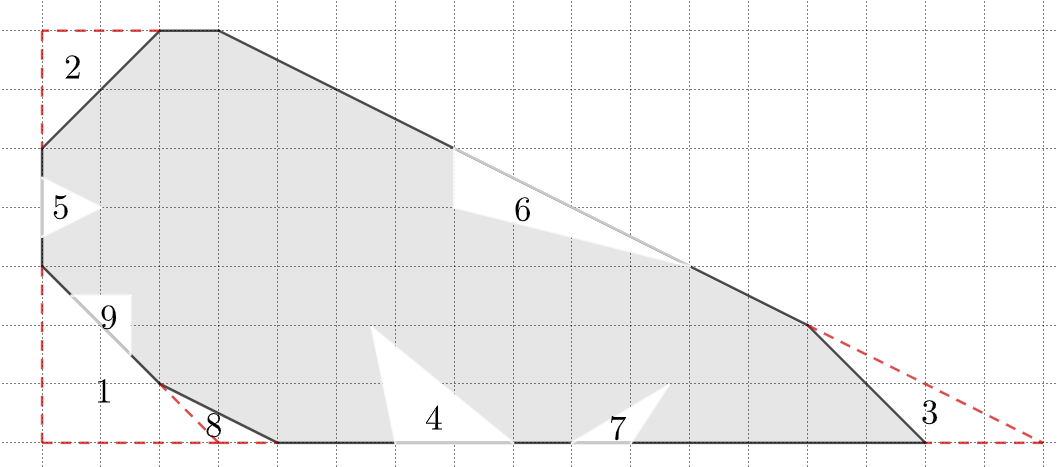}
 
		\caption{An example of a labeled regular bitten Delzant polygon. $P_0$ is a $2$-trapezoid and $P$ is obtained by performing four corner chopping. It is the almost toric base diagram of a symplectic rational manifold $(S^2\times S^2\# 9\overline{\CC\PP}^2,\omega)$. From the labels, we have a basis $\{b,f,e_1,\cdots,e_9\}$ in $H_2(S^2\times S^2\# 9\overline{\CC\PP}^2;\ZZ)$ with symplectic areas given by $\omega(f)=7,\omega(b)=10,\omega(e_1)=3,\omega(e_2)=\omega(e_3)=\omega(e_4)=\omega(e_6)=2,\omega(e_5)=\omega(e_7)=\omega(e_8)=\omega(e_9)=1$.\label{fig:delzant}}
	\end{figure}

Regular bitten Delzant polygons are insufficient to represent all LCY pairs. We also have to consider an enlarged class of almost toric bases by introducing two additional types of diagram. Let us begin with the following observations.

\begin{enumerate}[label=(\Roman*)]

    \item Note that a $k$-trapezoid has two opposite edges $X,X'$\footnote{Unfortunately, this notation coincides with the one we have used for the symplectic manifold. However the meaning should be clear in context.}, of equal affine length, corresponding to the $S^2$-fibers in the Hirzebruch surface. We now allow `{\bf full bites}' on $X$ or $X'$: that is, the sizes of the Symington triangles lying on these two edges are permitted to match the affine lengths of $X,X'$. The process of a full bite can be interpreted as filling in a corner of a regular bitten Delzant polygon, arising from an $\varepsilon${\bf -shift} of the edges $X,X'$ (see Figure \ref{fig:bd1}). By making a branch move as in Figure \ref{fig:examplefullbite}, one can readily see that this operation corresponds to a toric blowdown. Therefore, such a diagram with full bites still serves as an almost toric base.


    \item We can also allow the $2$-trapezoid to degenerate in the sense that its shortest edge has length $0$ (so that it becomes a triangle) with a single Symington triangle lying on the edge $X$ whose size equals the length of that edge. We can use the trick of a branch move to rotate around the node, allowing us to treat the situation similarly to Case (\upperRomannumeral{1}) by `filling up two small corners' of a regular bitten Delzant polygon. This process corresponds to performing two toric blowdowns of sizes $\varepsilon,\frac{\varepsilon}{2}$ respectively. Therefore, such a degenerate diagram with a full bite also serves as an almost toric base. We will call the regular bitten Delzant polygon an $(\varepsilon,\frac{\varepsilon}{2})${\bf -shift} of the original degenerate one. See Figure \ref{fig:bd2}.
\begin{figure}
		\includegraphics*[width=\linewidth]{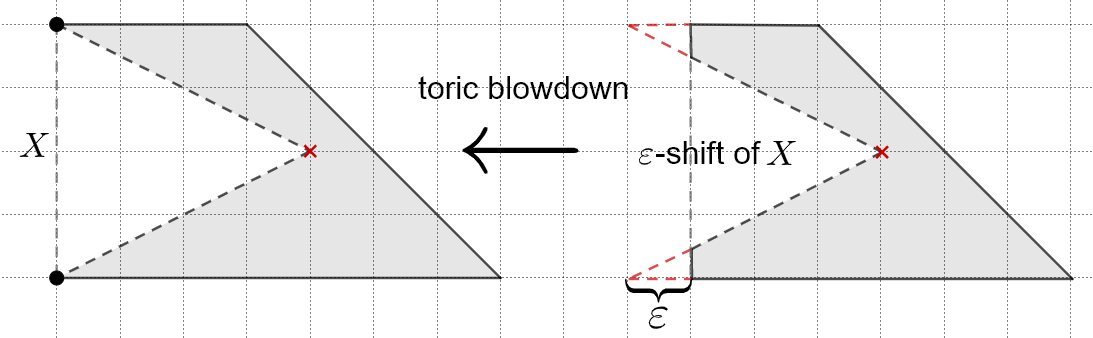}
 
		\caption{The right almost toric base diagram is a regular bitten Delzant polygon, obtained by an $\varepsilon$-shift of the edge $X$ from the left diagram. Consequently, the left diagram can be interpreted as the almost toric base diagram of the toric blowdown of the symplectic manifold represented on the right. The manifold on the right is $(\CC\PP^2\#2\overline{\CC\PP}^2,\omega')$ with homology basis $\{H,E_1,E_2\}$ and symplectic areas $\omega'(H)=6,\omega'(E_1)=2,\omega'(E_2)=3$; by blowing down the toric exceptional sphere in class $H_{12}$, the manifold on the left will be $(S^2\times S^2,\omega)$ with homology basis $\{b,f\}$ and symplectic areas $\omega(b)=3,\omega(f)=4$. Note that the two marked vertices in the left  diagram actually correspond to the same point—a rank $0$ elliptic singularity—in the ATF base disk.\label{fig:bd1}}
	\end{figure}
    
     \begin{figure}
		\includegraphics*[width=\linewidth]{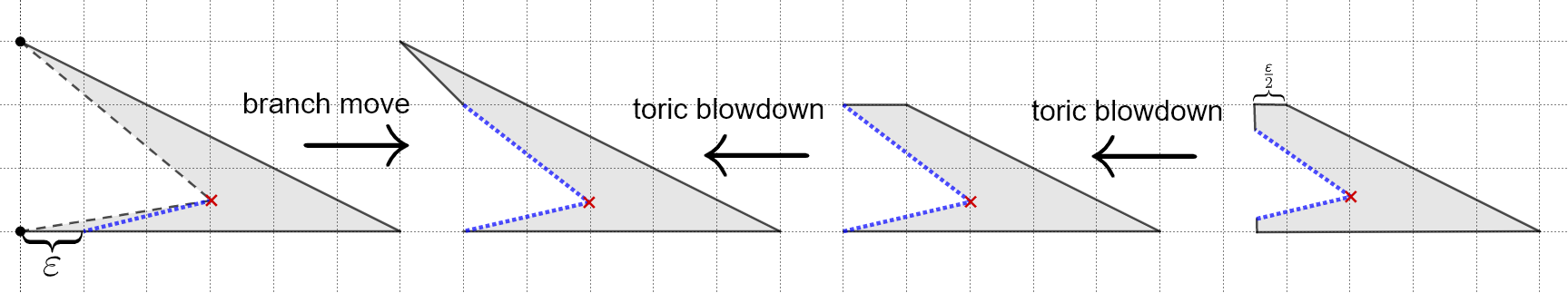}
 
		\caption{The leftmost diagram is a degenerate $2$-trapezoid with a full bite. By choosing a different branch cut (shown in blue), it can be interpreted as resulting from two toric blowdowns of sizes $\varepsilon,\frac{\varepsilon}{2}$ applied to the diagram of its $(\varepsilon,\frac{\varepsilon}{2})$-shift shown as the rightmost diagram. In this specific example, the manifold on the left is $(\CC\PP^2,\omega)$ with homology basis $\{H\}$ and symplectic area $\omega(H)=3$. It is obtained by first blowing down exceptional sphere of class $f-e$ in the manifold on the right $(S^2\times S^2\#\overline{\CC\PP}^2,\omega'')$ with homology basis $\{f,b,e\}$ and symplectic areas $\omega''(f)=3-\varepsilon,\omega''(b)=3-\frac{\varepsilon}{2},\omega''(e)=3-\frac{3\varepsilon}{2}$ ($\varepsilon$ is taken to be $1$ in this diagram), resulting the manifold $(\CC\PP^2\#\overline{\CC\PP}^2,\omega')$ with homology basis $\{H,E\}$ and symplectic areas $\omega'(H)=3,\omega'(E)=\frac{\varepsilon}{2}$; and then blowing down the exceptional sphere of class $E$.  As before, the marked vertices in the leftmost diagram represent the same point—a rank 0 elliptic singularity—in the ATF base disk.\label{fig:bd2}}
	\end{figure}
\end{enumerate}

Motivated by the above observations, we now introduce the following definitions. First, recall that any edge of a Delzant polygon has a neighborhood affine equivalent to
\[T_{u,v,w}:=\{(x,y)\in\RR^2\,|\,0\leq x\leq u-wy, 0\leq y\leq v\},\]
where $w\in \ZZ$ is the self-intersection of the symplectic sphere lying over the edge.

\begin{definition}
A {\bf special polygon} is a closed convex polygon in $\RR^2$ such that 
\begin{itemize}
    \item there is exactly one vertex whose neighborhood is spanned by $\begin{pmatrix}
        2\\1
    \end{pmatrix}$ and 
    $\begin{pmatrix}
        0\\1
    \end{pmatrix}$ up to some affine transformation, and all other vertices are Delzant;
    \item there is one edge whose neighborhood is affine equivalent to $T_{u,v,\frac{1}{2}}$.
\end{itemize} For convenience, we also refer to such an edge as having self-intersection $\frac{1}{2}$.
\end{definition}

\begin{definition}
     A {\bf special Symington triangle} is an embedded triangle in either a Delzant or special polygon $P$ such that one of its edge coincides with an edge $AB$ of $P$ with self-intersection $0$ when $P$ is Delzant or $\frac{1}{2}$ when $P$ is special, and the vertex opposite this edge has affine distance $l(AB)$ to $AB$. A {\bf Symington triangle} refers to either a regular or a special Symington triangle.
         
\end{definition}

\begin{definition}\label{def:bittendelzant}
    A {\bf special bitten Delzant polygon} consists of either a Delzant or special polygon $P$ together with a finite collection of disjoint regular and at least one speical Symington triangles embedded in $P$. A {\bf bitten Delzant polygon} refers to either a regular or a special bitten Delzant polygon.
\end{definition}

\begin{rmk}
    Note that in the above definition, we require all Symington triangles to be disjoint. In particular, this means that two full bites can not occur on two consecutive edges. By the observations discussed above, any bitten Delzant polygon naturally corresponds to a nodal integral affine disk, which serves as the almost toric base diagram of a certain symplectic rational surface. Let us also point out that full bites can actually be generalized in the following ways:
\begin{itemize}
    \item  we can allow two Symington triangles lying on an edge of self-intersection $1$ with sum of their sizes equal to that of the edge;
    \item if two consecutive edges have self-intersection $0$ and $-1$ respectively, we can allow a special Symington triangle on each of them.
\end{itemize}
See Figure \ref{fig:bd3} for why they also serve as almost toric base diagrams. In fact, the above items can be further generalized by allowing $w$ Symington triangles over an edge of self-intersection $w-1$, or considering chains of edges satisfying certain self-intersection conditions, each carrying a special Symington triangle. However, for the purposes of this paper—namely, realizing all LCY pairs—such generalizations are unnecessary, as Definition \ref{def:bittendelzant} is already sufficient. 
\end{rmk}
 \begin{figure}[h]
		\includegraphics*[width=\linewidth]{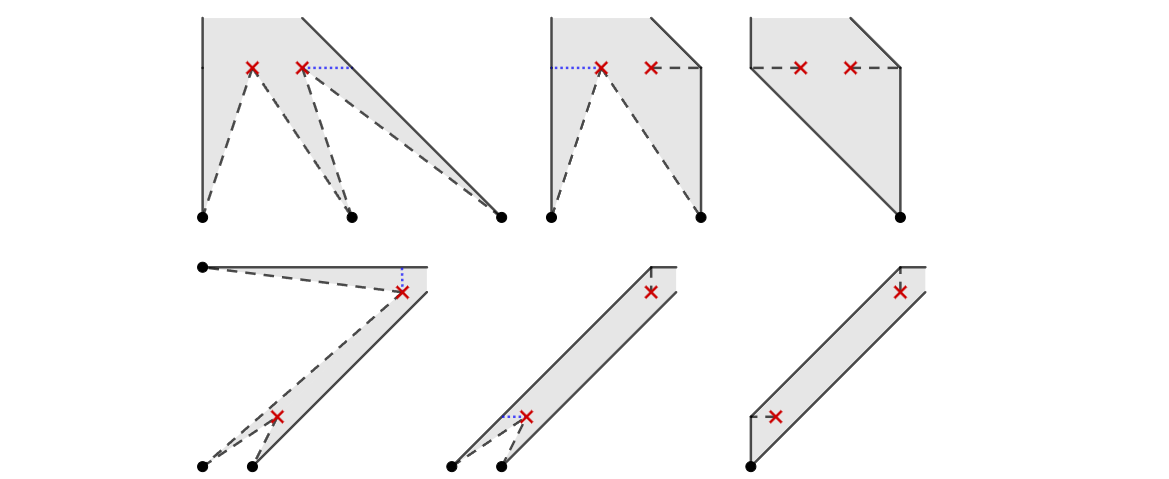}
		\caption{The top diagrams display two Symington triangles situated along an edge of self-intersection $1$. The bottom diagrams display two Symington triangles placed on consecutive edges of self-intersection $0$ and $-1$. By applying the trick of branch moves illustrated in the diagrams, one can see that they serve as valid almost toric bases for smooth manifolds.  \label{fig:bd3}}
	\end{figure}

\subsection{$\mathbb{BD}$ as almost toric base diagrams}

The equivalence relation of bitten Delzant polygons is defined as follows.

\begin{definition}
    Let $P$ be a bitten Delzant polygon with disjoint embeddings of Symington triangles $i=\bigsqcup i_\alpha:\bigsqcup T_\alpha\hookrightarrow P$. A a {\bf rearrangement of Symington triangles} is a combination of the following operations.
    \begin{itemize}
        \item (Sliding) For a smooth isotopy of embeddings of Symington triangle $i_{\alpha}(t): T_\alpha\hookrightarrow P$ such that $i_\alpha(0)=i_\alpha$ and $i_\alpha(1)$ is disjoint from all other Symington triangles, then we can replace $i_\alpha$ by $i_\alpha(1)$.
        \item (Swapping) For two regular Symington triangles $T_\alpha,T_\beta$ lying on the same edge $AB\subseteq P$, we can replace $i_\alpha\bigsqcup i_\beta$ by another regular Symington triangle embedding $i'_\alpha\bigsqcup i'_\beta$ based on the edge $AB$ and disjoint from all other Symington triangles, which interchanges their positions.
    \end{itemize}
\end{definition}
Intuitively, just as each Delzant polygon uniquely determines a symplectic toric manifold up to (equivariant) symplectomorphisms, a bitten Delzant polygon equipped with labels on all chopped corners and Symington triangles can be associated to an almost toric fibration of a symplectic rational manifold together with a framing. Rearrangement of the Symington triangles does not alter the symplectomorphism type of the resulting manifold. This will be made more precise in the following Definition \ref{def:BD} and Proposition \ref{prop:bdtolcy}. 

\begin{definition}\label{def:BD}
      Two bitten Delzant polygons are said to be equivalent if they differ only by rearrangement of Symington triangles and an $\text{AGL}(2,\ZZ)$-transformation. We define $\mathbb{BD}$ to be the set of all equivalence classes of bitten Delzant polygons. 
\end{definition}

  As what we did in section \ref{section:lcy} for $\mathbb{LCY}$, the set $\mathbb{BD}$ can be similarly endowed with a quiver structure by assigning an arrow between two elements of $\mathbb{BD}$ whenever there exist representatives that differ by either chopping a corner or adding a regular Symington triangle. Now, let us carefully explain the relation between $\mathbb{BD}$ and $\mathbb{LCY}$ through almost toric fibrations. 

\begin{prop}\label{prop:bdtolcy}
There is a canonical morphism between quivers $\mathbf{B}:\mathbb{BD}\rightarrow \mathbb{LCY}$ by taking the boundary divisor.
 \end{prop}
 
\begin{proof}
    At the level  prior to taking the quotient by the equivalence relations in both 
 $\mathbb{BD}$ and $\mathbb{LCY}$, we have observed that each bitten Delzant polygon determines a stratified integral affine disk with nodes $(B,\mathcal{A},\mathcal{S})$ in the sense of \cite{symington}. According to Theorem 5.2 and Corollary 5.4 in \textit{loc. cit.}, one can construct an almost toric fibered symplectic four-manifold over the base $(B,\mathcal{A},\mathcal{S})$. By Proposition \ref{prop:symington}, the fibers over $0$ and $1$-strata will form a divisor whose homology class is Poincar\'e dual to the first Chern class. Thus we will obtain a symplectic log Calabi-Yau pair. 

 To upgrade this assignment into a well-defined map $\mathbf{B}:\mathbb{BD}\rightarrow \mathbb{LCY}$, we make use of the following observations. Let us first assume the base diagram is a regular bitten Delzant polygon. For any regular Symington triangle lying on some edge of the Delzant polygon, \cite[Theorem 7.4]{symington} ensures the existence of a visible embedded symplectic sphere (unique up to isotopy) lying over the curve (with tangent vectors not parallel to the edge) going from the node to the edge. This sphere meets the boundary LCY divisor in exactly one positively transverse point and therefore has self-intersection number $-1$ by the adjunction formula. Blowing down all such exceptional spheres arising from Symington triangles yields a toric LCY pair, as verified by a direct computation of the change in the number of divisor components and $b_2$. The area sequence and self-intersection sequence of this toric pair remain unchanged under any rearrangement of Symington triangles in the bitten Delzant polygon. Thus, by tautness of toric pairs (Lemma \ref{lem:torictaut}), we will obtain a well-defined toric element in $\mathbb{LCY}$. Then, by Torelli Theorem  \ref{prop:Torelli} and the followed illustration therein, the original log Calabi-Yau pair assigned to the regular bitten Delzant polygon will also represent a uniquely determined element in $\mathbb{LCY}$. 
 
 Next, suppose the base diagram is a special bitten Delzant polygon. As explained in Section \ref{sec:bitten}, we can choose sufficiently small $\varepsilon>0$ and consider its $\varepsilon$ or $(\varepsilon,\frac{\varepsilon}{2})$-shift, which yields a regular bitten Delzant polygon. By the argument in the preceding paragraph, the shift determines a well-defined element in $\mathbb{LCY}$. The original LCY pair, associated to the speical bitten Delzant polygon, is then obtained from that element by performing the toric blowdowns with sizes $\varepsilon$ or $\varepsilon,\frac{\varepsilon}{2}$. When $\varepsilon$ is small enough, the exceptional spheres to be blown down are uniquely identifiable among the divisor components by their symplectic areas, which enables us to match the divisor components after blowdown for different representatives of a class in $\mathbb{BD}$. Therefore, we will still get a uniquely determined element in $\mathbb{LCY}$ again by Torelli Theorem \ref{prop:Torelli}.
    
    To further upgrade this map into a morphism between quivers, first observe that the arrow corresponding to chopping a corner in $\mathbb{BD}$ naturally corresponds to the arrow representing a toric blowup in $\mathbb{LCY}$. To establish the correspondence between adding a regular Symington triangle and a non-toric blowup, we blow down the exceptional sphere associated to the newly added Symington triangle. It remains to show that the pair obtained after the non-toric blowdown represents the same element in $\mathbb{LCY}$ as the one associated to the base diagram before adding the regular Symington triangle. To see this, one can run a similar argument in the previous two paragraphs: for regular (resp. special) bitten Delzant polygons, by further blowing down the exceptional spheres associated to all the other regular Symington triangles (resp. after the $\varepsilon$ or $(\varepsilon,\frac{\varepsilon}{2})$-shift), their resulting toric pairs will have the same area sequence and self-intersection sequence. Again, the tautness of toric pairs, together with the Torelli theorem, ensures the desired conclusion.
\end{proof}

With this understood, we can give a precise statement of our main result on ATF realizations for LCY in rational surfaces. The proof will be deferred to Section \ref{section:proof}.

\begin{theorem}\label{thm:main}
   Given a framing $\mathfrak{f}$ on $\mathbb{LCY}$, there is a canonical set-theoretic section $\mathbf{R}_{\mathfrak{f}}:\mathbb{LCY}_{\geq2}\rightarrow \mathbb{BD}$ of the map $\mathbf{B}:\mathbb{BD}\rightarrow \mathbb{LCY}$.
\end{theorem}




\section{Proof of the main results}\label{section:proof}

\subsection{Elementary geometry of Delzant polygons}\label{section:elementarygeometry}

We begin by establishing some elementary geometric properties of Delzant polygons, which will serve as preparation for the proof of Theorem \ref{thm:main}. Let $P$ be a Delzant polygon in $\RR^2$ with edges $Q_1,\cdots,Q_n$. Given positive real numbers $a_1,\cdots,a_n$ satisfying $a_i<l(Q_i)$ for each $i$, we define the {\bf triangle packing problem with weights $a_1,\cdots,a_n$} as follows:
\begin{itemize}
    \item determine whether there exists an interior point $O$ in $P$ such that $d(O,Q_i)\geq a_i$ for all $i$'s.
\end{itemize} If such a point exists, we say the problem is {\bf solvable} and refer to $O$ as a {\bf solution point}. By Proposition \ref{prop:bdtolcy}, the existence of such a point corresponds to the realization of non-toric blowups of with sizes less than the specified weights.

 We now focus on a specific setting that will play a central role in our subsequent constructions. Let $P_0$ be a $k$-trapezoid whose edges $X,Y,X',Z$ have affine lengths $x,y,x,z$ respectively, where $y\geq z$. Suppose $P_1$
  is the Delzant polygon obtained from $P_0$ by performing a corner chopping procedure. Each corner chopping is determined by the choice of a vertex on the previous Delzant polygon and the size (maximal affine distance between two points in the triangle) of the chopped triangle. For the vertices $X\cap Y$, $Y\cap X'$, $X'\cap Z$ and $Z\cap X$ on the initial trapezoid, define $c_{X\cap Y}$, $c_{Y\cap X'}$, $c_{X'\cap Z}$ and $c_{Z\cap X}$ to be $0$ if the corresponding vertex is not chosen in the corner chopping procedure; or the size of the chopped triangle at the corresponding vertex. Then, we can introduce the quantities
  \[b_X:=\max\{c_{X\cap Y},c_{Z\cap X}\},b_Y:=\max\{c_{X\cap Y},c_{Y\cap X'}\},\]
  \[b_{X'}:=\max\{c_{Y\cap X'},c_{X'\cap Z}\},b_{Z}:=\max\{c_{X'\cap Z},c_{Z\cap X}\}.\]

 See Figure \ref{fig:trianglepacking} for an example. Next, denote by $X_1,Y_1,X'_1,Z_1$ the edges in $P_1$ corresponding to the portions of the original four edges $X,Y,X',Z$ in $P_0$ that remain after the corner chopping. Suppose there is a triangle packing problem on $P_1$ with weights $a_{X_1},a_{Y_1},a_{X'_1},a_{Z_1}$ on the edges $X_1,Y_1,X'_1,Z_1$. We further introduce the following quantities:
 \[\Theta_1:=\max\{a_{Y_1},b_{Y}\}+\max\{a_{Z_1},b_{Z}\},\]
 \[\Theta_2:=\max\{a_{X_1},b_{X}\}+\max\{a_{X'_1},b_{X'}\}+k\max\{a_{Y_1},b_{Y}\}.\]

 \begin{figure}[h]
		\includegraphics*[width=\linewidth]{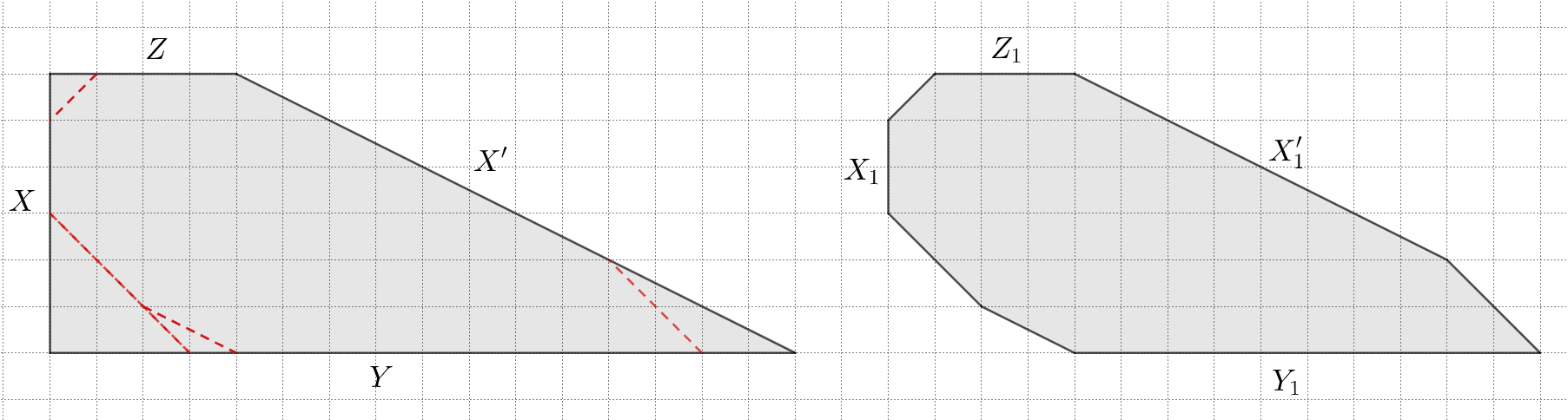}
 
		\caption{The left polygon is a $2$-trapezoid $P_0$ with $x=6,y=16,z=4$. After truncating the corners indicated by the red segments, we obtain the right polygon $P_1$. In this case, $b_X=\max\{1,3\}=3,b_Y=\max\{3,2\}=3,b_{X'}=\max\{2,0\}=2,b_{Z}=\max\{0,1\}=1$.\label{fig:trianglepacking}}
	\end{figure}

With the help of Figure \ref{fig:trianglepacking}, the geometric meaning of $\Theta_1,\Theta_2$ can be interpreted as follows, making the ansatz that the $P_1$ becomes a regular bitten Delzant polygon.
\begin{itemize}
    \item If we projective all the Symington triangles lying over the edges $Y_1,Z_1$ together with all the triangles chopped at the initial four vertices of $P_0$ onto a vertical line, then $\Theta_1$ is the affine length of their projection.
    \item The term $\max\{a_{X_1},b_{X}\}+\max\{a_{X'_1},b_{X'}\}$ in $\Theta_2$ can be interpreted analogously to  $\Theta_1$: it is the affine length of the projection of all the Symington triangles lying over the edges $X_1,X'_1$ together with all the triangles chopped at the initial four vertices of $P_0$ onto a horizontal line. The term $k\max\{a_{Y_1},b_{Y}\}$ in $\Theta_2$ can be understood as follows: first project all the Symington triangles lying on $Y_1$ and the triangles chopped at $X\cap Y,X'\cap Y$ onto the edge $X'$ along the horizontal direction, then project the portion in $X'$ onto the horizontal line along the vertical direction; finally take the affine length of this projection.
\end{itemize}

\begin{lemma}\label{lemma:packing}
    The triangle packing problem in the above setting is solvable if $\Theta_1\leq x$ and $\Theta_2\leq y$.
\end{lemma}

\begin{proof}
By the discussions above, the first condition $\Theta_1\leq x$ implies that there exists a horizontal line between $Y_1$ and $Z_1$ whose affine distance to $Y_1$ (resp. $Z_1$) is $\max\{a_{Y_1},b_{Y}\}$ (resp. at least $\max\{a_{Z_1},b_{Z}\}$). The second condition $\Theta_2\leq y$ then ensures the existence of a point $O$ on the intersection of this horizontal line with $P_1$ such that $O$ serves as a solution point to the four edges $X_1,Y_1,X_1',Z_1$. To verify $O$ is also the solution point for the other edges arising from corner chopping, we make use of the following observation. Assume a point in the Delzant polygon has affine distance at least $d$ to two adjacent edges. If we truncate the corner of size at most $d$ between those two edges, then the affine distance from the point to the new edge must be at least $d$. Since each weight $a_i$ is strictly less than the affine length of its corresponding edge, and in our definitions of $\Theta_1,\Theta_2$ we have taken the maximum with those sizes $b_*$'s of the chopped triangles, the point $O$ is indeed the solution point for the entire triangle packing problem by induction on the number of corner choppings.
\end{proof}

Unfortunately, the existence of such a point does not immediately give the disjoint regular Symington triangles we expect since they might overlap at the point $O$. See Figure \ref{fig:perturbation} for an example about such an issue. We need the following triangle perturbation lemma to solve this issue.

 \begin{figure}[h]
		\includegraphics*[width=\linewidth]{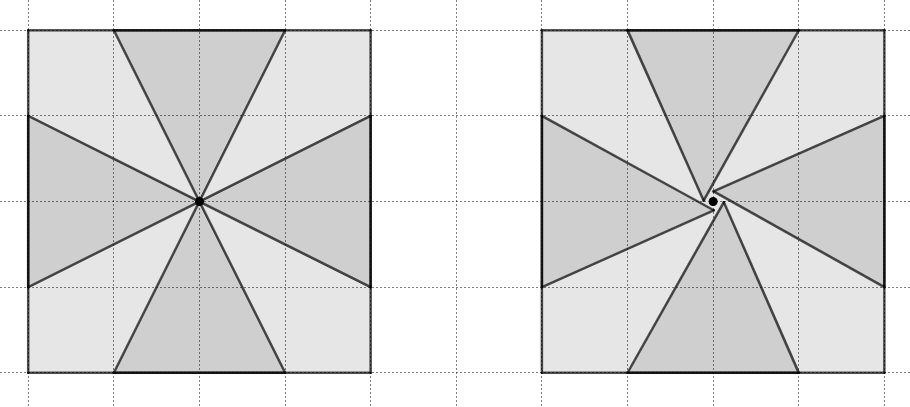}
 
		\caption{Suppose the Delzant polygon in the triangle packing problem is a square whose edges have affine length $4$ and the weights $a_1,a_2,a_3,a_4$ are $2,2,2,2$. There is a unique point solving this problem. However, it does not immediately yield the desired almost toric blowup diagram. Some perturbations are necessary to achieve the correct diagram.
\label{fig:perturbation}}
	\end{figure}

\begin{lemma}[{\bf Triangle perturbation}]\label{lemma:embed}
Suppose there is a solvable triangle packing problem with weight $a_i$ on the edge $Q_i$. Assume for each $i$, there is a finite sequence $\{b_{ij}\}_j$ of positive numbers such that $b_{ij}\leq a_i$ and $\sum_jb_{ij}<l(Q_i)$. Then, one can always find a disjoint embedding of regular Symington triangles of sizes $b_{ij}$ on the edge $Q_i$.
\end{lemma}
\begin{proof}
    We prove only the case when each $\{b_{ij}\}_j$ is simply $\{a_i\}$. The general case follows analogously, requiring only more elaborate notation. Assume each edge has primitive direction $\overrightarrow{p_i}$ (in a counterclockwise way) and the point $O$ is the solution point to the triangle packing problem. We wish to shift each regular Symington triangle lying on $Q_i$ along the vector $\varepsilon_i\overrightarrow{p_i}$ for some small $\varepsilon_i$. Denote by $O_i:=O+\varepsilon_i\overrightarrow{p_i}$ the vertex of the shifted Symington triangle. By applying a suitable $\text{AGL}(2;\ZZ)$-transformation, we can focus on the corner $R$ between the edges $Q_{i-1}$ and $Q_i$, as illustrated in Figure \ref{fig:triangleshift}. Note that we have slopes $s_{OB}>s_{OR}>s_{OA}$. To ensure that the shifted regular Symington triangles do not overlap, it suffices to check that $O_{i-1}$ does not lie below the segment $O_iB'$, which will be guaranteed by $\frac{\varepsilon_{i-1}}{\varepsilon_{i}}<s_{OB}$. Hence, it also suffices to have $\frac{\varepsilon_{i-1}}{\varepsilon_{i}}\leq s_{OR}=\frac{a_i}{a_{i-1}}$. Now, by choosing $\varepsilon\ll l(Q_i)$ for all $i$'s, and setting $\varepsilon_i:=\frac{a_1}{a_i}\varepsilon$, this condition will be satisfied.
\end{proof}

\begin{figure}[h]
		\includegraphics*[width=\linewidth]{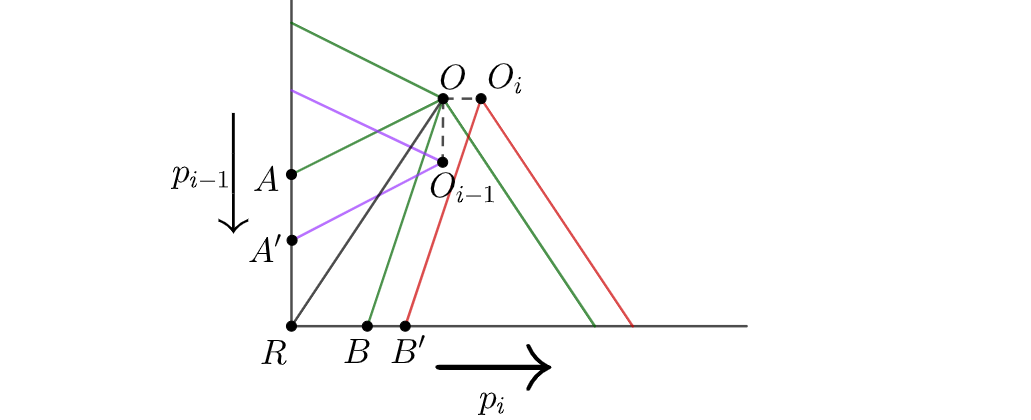}
 \caption{The process of repositioning regular Symington triangles so as to eliminate the overlap of nodal singularities.}\label{fig:triangleshift}
	\end{figure}

\begin{rmk}
It is observed in \cite[Lemma 5.6]{Engel2021} that one can simultaneously perform nodal slides to eliminate the overlap. In our proof of Lemma \ref{lemma:packing}, we take a slightly different approach, allowing the bottom edge of the Symington triangles to be shifted. This approach makes the argument more straightforward to present as above.
\end{rmk}

\subsection{Proof of the main Theorem}
We now aim to canonically assign an almost toric realization $\mathbf{R}_{\mathfrak{f}}(\XX)\in\mathbb{BD}$ for any $\XX=[X,\omega,D]\in\mathbb{LCY}_{\geq 2}$. Suppose $X=\CC\PP^2\#n\overline{\CC\PP}^2$ with $n\geq 2$ and the framing $\mathfrak{f}$ on $\mathbb{LCY}$ gives a basis $\{H,E_1,\cdots,E_n\}\subseteq H_2(X;\ZZ)$ with $\omega(H)=c,\omega(E_i)=\dd_i$. Without loss of generality, let us also assume the normalized condition $c=1$. Then, the reduced condition yields the following key inequalities which we will frequently use: 
\[\dd_1\geq\dd_2\geq\cdots\geq\dd_n, \,\,\, 1\geq\dd_1+\dd_2+\dd_3.\]
By the discussions in Section \ref{section:lcy}, let us choose the canonical $\varepsilon$-replacement $\XX_{\varepsilon}$ and the symplectic toric model $\XX_{\text{tor}}$ along with the morphism $\Gamma_n\in\text{Hom}(\XX_{\text{tor}},\XX_{\varepsilon})$ consisting of only non-toric blowups. There is also the morphism $\Gamma_t\in\text{Hom}(\XX_{\text{tor}}',\XX_{\text{tor}})$ consisting of only toric blowups, where $\XX_{\text{tor}}'$ is given by the boundary divisor of a symplectic toric Hirzebruch surface with moment polygon $P_0$

\begin{proof}[Proof of Theorem \ref{thm:main}]

The blowup patterns $\Gamma_n$ and $\Gamma_t$ naturally give us a triangle packing problem in the following manner. Starting from the trapezoid $P_0$, one can derive a corner chopping pattern from the toric blowups in $\Gamma_t$. This process produces the Delzant polygon $P_1$ after the corner chopping. The weights of triangle packing problem assigned to the edges of $P_1$ are determined by the non-toric blowups in $\Gamma_n$; specifically, each weight is the maximum size among all non-toric blowups occurring on the corresponding edge.\footnote{For greater rigor, we require all weights to be positive; otherwise, the solution point may lie on the boundary of the Delzant polygon. Some edges may not admit non-toric blowups, in which case we assign a very small positive weight $a\ll \dd_i$ for all $i$'s. The subsequent inequality estimates remain valid under this convention. } Note that, although the non-toric blowup pattern $\Gamma_n$ is not canonical, any alternative choice differs only by the order of blowups and therefore yields the same triangle packing problem.

Let us show it is solvable by checking the conditions of Lemma \ref{lemma:packing} (details are fully spelled out for the first case). Let $\dd_i,\dd_j$ be two distinguished blowup sizes introduced in Section \ref{subsection:replace}. Recall that they are either among $\dd_1,\cdots,\dd_n$ when $\XX_{\varepsilon}=\XX$; or significantly smaller than $\dd_n$ when $\XX_{\varepsilon}\neq \XX$. Also, let $f_1$ denote the first blowup in the canonically assigned morphism $\Gamma\in\text{Hom}(\XX_{\text{red}},\XX)$.
\begin{enumerate}[label=(\roman*)]
	\item If $f_1$ is the non-toric blowup on $2H$, then $\XX_{\text{tor}}'$ is given by the trapezoid with $(k,x,y,z)=(1,1-\dd_2,2-\dd_2-\dd_i-\dd_j,1-\dd_i-\dd_j)$. By the assumption on $f_1$, we have $b_X,b_{X'},b_{Y},b_Z\leq \dd_3$ (actually must be $\leq\dd_5$ since the exceptional class $E_2$ is non-toric, $E_3,E_4$ are either non-toric or distinguished). According to our choice of $\Gamma_n$, we also have $a_{X_1}\leq 1-\dd_2-\dd_i,a_{X_1'}\leq 1-\dd_2-\dd_j,a_{Y_1}\leq \dd_3$ and $a_{Z_1}=\dd_1$. Therefore 
 $$\Theta_1\leq \dd_1+\dd_3\leq 1-\dd_2=x,$$
 $$\Theta_2\leq (1-\dd_2-\dd_i)+(1-\dd_2-\dd_j)+\dd_3=y+(\dd_3-\dd_2)\leq y;$$

 Otherwise, $\XX_{\text{tor}}'$ is given by the trapezoid with $(k,x,y,z)=(2,1-\dd_i,2-\dd_i-\dd_j,\dd_i-\dd_j)$. Note that we still have $b_X,b_{X'},b_{Y},b_Z\leq \dd_3$ (actually must be $\leq\dd_4$ since the exceptional classes $E_2,E_3$ are either non-toric or distinguished). Also $a_{X_1}\leq \dd_1,a_{X_1'}\leq 1-\dd_i-\dd_j,a_{Z_1}\leq \dd_2$ and $a_{Y_1}\leq\dd_3$ by the assumption on $f_1$. Therefore
 $$\Theta_1\leq \dd_2+\dd_3\leq 1-\dd_1\leq 1-\dd_i =x,$$
 $$\Theta_2\leq \dd_1+(1-\dd_i-\dd_j)+2\dd_3\leq 2-\dd_1-\dd_2=y.$$

 \item If $f_1$ is the non-toric blowup on $(a+1)H-aE_1$,
\[\Theta_1\leq \dd_3+(1-\dd_1-\dd_2)\leq 1-\dd_1=x,\]
 $$\Theta_2\leq (1-\dd_1-\dd_i)+(1-\dd_1-\dd_j)+(2a-2)\dd_3=2-\dd_i-\dd_j+(a-1)(\dd_1+2\dd_3)-a\dd_1-\dd_1 $$
     $$\leq 2-\dd_i-\dd_j+a-1-a\dd_1-\dd_2=y;$$

 otherwise, $a_{Y_1},b_{Y_1}\leq \dd_3$ holds and we have 
 $$\Theta_1\leq \dd_2+\dd_3\leq 1-\dd_1=x,$$
 $$\Theta_2\leq (1-\dd_1-\dd_i)+(1-\dd_1-\dd_j)+(2a-1)\dd_3=(2-\dd_i-\dd_j)+(a-2)(\dd_1+2\dd_3)+3\dd_3-a\dd_1$$
 $$\leq (2-\dd_i-\dd_j)+a-2+1-a\dd_1=y.$$

 

 

 \item If $f_1$ is the non-toric blowup on $aH-(a-1)E_1$,
\[\Theta_1\leq \dd_3+(1-\dd_1-\dd_2)\leq 1-\dd_1=x,\]
 $$\Theta_2\leq (1-\dd_1-\dd_i)+\dd_3+(2a-3)\dd_3=1-\dd_i+(a-1)(\dd_1+2\dd_3)-(a-1)\dd_1-\dd_1$$
     $$\leq 1-\dd_i+a-1-(a-1)\dd_1-\dd_2=y;$$

if $f_1$ is the toric blowup on the node between $aH-(a-1)E_1$ and $H-E_1$, then $a_{Y_1},b_{Y_1}\leq \dd_3$, $a_{X_1},b_{X_1}\leq 1-\dd_1-\dd_2$, $a_{X_1'},b_{X_1'}\leq 1-\dd_1-\dd_i$ hold and we have
\[\Theta_1\leq \dd_3+(1-\dd_1-\dd_2)\leq 1-\dd_1=x,\]
 $$\Theta_2\leq (1-\dd_1-\dd_i)+(1-\dd_1-\dd_2)+(2a-3)\dd_3=2-\dd_2-\dd_i+(a-2)(\dd_1+2\dd_3)-(a-1)\dd_1+(\dd_3-\dd_1)$$
     $$\leq 2-\dd_2-\dd_i+a-2-(a-1)\dd_1=y;$$
     
 otherwise, $a_{Y_1},b_{Y_1}\leq \dd_3$ still holds and we have 
 $$\Theta_1\leq \dd_2+\dd_3\leq 1-\dd_1=x,$$
 $$\Theta_2\leq (1-\dd_1-\dd_i)+\dd_2+(2a-2)\dd_3=1-\dd_i+(a-1)(\dd_1+2\dd_3)-(a-1)\dd_1+(\dd_2-\dd_1)$$
 $$\leq 1-\dd_i+a-1-(a-1)\dd_1=y.$$
 
 
 
 

\item If $f_1$ is the non-toric blowup on $aH-(a-1)E_1$,
\[\Theta_1\leq \dd_3+(1-\dd_1-\dd_2)\leq 1-\dd_1=x,\]
 $$\Theta_2\leq \dd_3+\dd_4+(2a-2)\dd_3=a(\dd_1+2\dd_3)-(a-1)\dd_1-\dd_2+(\dd_2-\dd_1)+(\dd_4-\dd_3)$$
     $$\leq a-(a-1)\dd_1-\dd_2=y;$$

if $f_1$ is the toric blowup on the node between $aH-(a-1)E_1$ and $H-E_1$, then $a_{Y_1},b_{Y_1}\leq \dd_3$, $a_{X_1},b_{X_1}\leq 1-\dd_1-\dd_2$, $a_{X_1'},b_{X_1'}\leq \dd_3$ hold and we have
\[\Theta_1\leq \dd_3+(1-\dd_1-\dd_2)\leq 1-\dd_1=x,\]
 $$\Theta_2\leq (1-\dd_1-\dd_2)+\dd_3+(2a-2)\dd_3=1-\dd_2+(a-1)(\dd_1+2\dd_3)-(a-1)\dd_1+(\dd_3-\dd_1)$$
     $$\leq 1-\dd_2+a-1-(a-1)\dd_1=y;$$

     otherwise, $a_{Y_1},b_{Y_1}\leq \dd_3$ still holds and we have
     $$\Theta_1\leq \dd_2+\dd_3\leq 1-\dd_1=x,$$
 $$\Theta_2\leq \dd_2+\dd_3+(2a-1)\dd_3=a(\dd_1+2\dd_3)-(a-1)\dd_1+(\dd_2-\dd_1)$$
 $$\leq a-(a-1)\dd_1=y.$$
 
 

 
 

 \item $$\Theta_1\leq \dd_3+(1-\dd_2-\dd_l)\leq 1-\dd_l=x,$$
 $$\Theta_2\leq (1-\dd_l-\dd_i)+(1-\dd_l-\dd_j)+\dd_3=y-(1-\dd_1-\dd_2-\dd_3)\leq y.$$
 
 
    
 
 
 
\end{enumerate}


 Therefore, Lemma \ref{lemma:packing} guarantees a disjoint embedding of all the regular Symington triangles and we obtain a well-defined element $\mathbb{P}$ in $\mathbb{BD}$. When $\XX_{\varepsilon}$ is equal to $\XX$, $\mathbf{R}_{\mathfrak{f}}(\XX)$ is then defined to be this $\mathbb{P}$. Otherwise, according to our construction of $\XX_{\varepsilon}$, there will be one or two (depending on the number of small blowups performed to get $\XX_{\varepsilon}$) distinguished edges of the bitten Delzant polygon $\mathbb{P}$ whose affine lengths are significantly smaller than those of the other edges. As explained in Section \ref{sec:bitten}, there will be a special bitten Delzant polygon $\mathbb{P}'\in\mathbb{BD}$ whose $\varepsilon$-shift (or $(\varepsilon,\frac{\varepsilon}{2})$-shift) is $\mathbb{P}$, in which case $\mathbf{R}_{\mathfrak{f}}(\XX)$ is defined to be this $\mathbb{P}'$. 
 
 Finally, to verify that $\mathbf{B}(\mathbf{R}_{\mathfrak{f}}(\XX))=\XX$, first notice that $\mathbf{B}(\mathbb{P})$ always yields $\XX_{\varepsilon}$ by our construction of $\XX_{\text{tor}}$ in Section \ref{section:toricmodel} and the map $\mathbf{B}$ in Proposition \ref{prop:bdtolcy}. Thus, if $\XX=\XX_{\varepsilon}$ the claim follows immediately. Otherwise, as explained in Section \ref{sec:bitten}, since the inverse of $\varepsilon$-shift (or $(\varepsilon,\frac{\varepsilon}{2})$-shift) corresponds to the $\varepsilon$-replacement $\Gamma_\varepsilon:\XX\rightarrow\cdots\rightarrow \XX_{\varepsilon}$, $\mathbf{B}(\mathbb{P}')$ will be exactly $\XX$. The proof of Theorem \ref{thm:main} is now completed.
 
 \end{proof}
 
 \begin{corollary}\label{maincor}
    Every symplectic rational manifold $(X,\omega)$ with $[\omega]\cdot c_1(X,\omega)>0$ admits an almost toric fibration.
\end{corollary}

\begin{proof}
We may assume $b_2(X)\geq 3$ since the case when $b_2(X)\leq 2$ is well-known. Choose a framing $\{H,E_1,\cdots,E_{l-1}\}\subseteq H_2(X;\ZZ)$ such that the condition $[\omega]\cdot c_1(X,\omega)>0$ becomes \[\omega(E_1)+\cdots+\omega(E_{l-1})<3\omega(H).\] If we take the symplectic blowup of $(X,\omega)$ with sufficiently small size, there will be $(X\#\overline{\CC\PP}^2,\omega')$ with \[\omega'(E_1)+\cdots+\omega'(E_{l-1})+2\omega'(E_{l})<3\omega'(H).\]  One can then assume that there exists a LCY divisor in $(X\#\overline{\CC\PP}^2,\omega')$ with homology sequence $(3H-E_1-\cdots-E_{l-1}-2E_{l},E_{l})$ by \cite[Proposition 2.24]{Enumerate}. This is a symplectic reduced model of type (\rn{5}). Theorem \ref{thm:main} then gives an almost toric realization of such a LCY pair presented by a bitten Delzant polygon. By applying several branch moves shown in Figure \ref{fig:nodal}, we can arrange the codimension $1$-stratum representing the component $E_{l}$ to form a segment of self-intersection $-1$, after which we fill in the corner. The resulting diagram will correspond to the toric blowdown of the exceptional class $E_l$. This gives an almost toric base diagram for $(X,\omega)$ since it is well known that the space of symplectic spheres in class $E_{l}$ is connected (\cite[Theorem B]{Wendl}). 
\end{proof}

\begin{rmk}
   The almost toric base diagram obtained in Corollary \ref{maincor} actually corresponds to a `uninodal Looijenga pair' $(X,\omega,D)$ in the sense that $D$ is an immersed symplectic sphere with a single positive self-intersection point representing $\text{PD}(c_1(X,\omega))$. 
\end{rmk}

\begin{figure}[h]
		\includegraphics*[width=\linewidth]{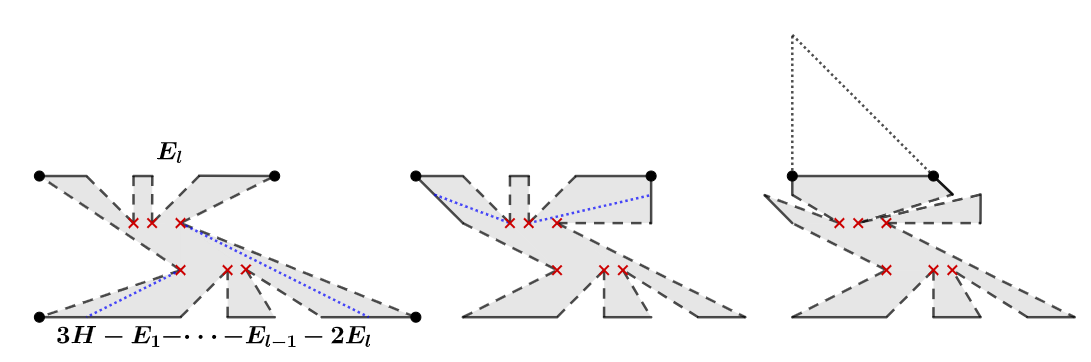}
 \caption{The process of blowing down the $E_l$-component for a type (\rn{5}) symplectic reduced model. The leftmost diagram shows the bitten Delzant polygon, where the bites on the top (resp. bottom) edges denote classes $H_{1l}, H_{2l}$ (resp. $E_3,\cdots,E_{l-1}$). We refer the reader to Figure \ref{fig:toric5} for comparison. After several branch moves, we arrive at the rightmost diagram which contains a segment of self-intersection $-1$. By filling it with a triangle, we then obtain an almost toric base for $(X,\omega)$. For the beauty of these diagrams, the length of $E_l$
  is exaggerated, although in our setting $E_l$
  is assumed to have sufficiently small symplectic area.}\label{fig:nodal}
	\end{figure}

\subsection{Further discussions}\label{section:EF21}

\begin{itemize}

\item {\bf (Dependence of the framing)} Let us revisit Example \ref{exa:framing} to see how the choice of framing affects the resulting associated bitten Delzant polygon. Given the framing $\mathfrak{f}_1$ with the reduced model $\XX_{\text{red}}^{(1)}$, the $\varepsilon$-replacement $\XX^{(1)}_{\varepsilon}$ is obtained by performing a small toric blowup at $C_H\cap C_{H_3}$ and $\XX^{(1)}_{\text{tor}}=\XX^{'(1)}_{\text{tor}}$ has a rectangle as its moment polygon. To obtain $\mathbf{R}_{\mathfrak{f}_1}(\XX)$, we need to pack one special and two regular Symington triangles. On the other hand, given the framing $\mathfrak{f}_2$ with the reduced model $\XX_{\text{red}}^{(2)}$, the $\varepsilon$-replacement $\XX^{(2)}_{\varepsilon}$ is obtained by performing a small toric blowup at $C_{2h-e_1-e_2-e_3}\cap C_{h_3}$ and $\XX^{(2)}_{\text{tor}}=\XX^{'(2)}_{\text{tor}}$ has a $1$-trapezoid as its moment polygon. To obtain $\mathbf{R}_{\mathfrak{f}_2}(\XX)$, we again need to pack one special and two regular Symington triangles. See Figure \ref{fig:dependframing} for their presentations, where we can see that $\mathbf{R}_{\mathfrak{f}_1}(\XX),\mathbf{R}_{\mathfrak{f}_2}(\XX)$ indeed represent different elements in $\mathbb{BD}$. However, these two bitten Delzant polygons correspond to essentially the same almost toric fibrations, related by a branch move as shown in Figure \ref{fig:dependframing}. This motivates the following speculation: if we further consider the map $\mathbf{I}:\mathbb{BD}\rightarrow \{(B,\mathcal{A},\mathcal{S})\}/\sim$ which assigns to each bitten Delzant polygon the isomorphism class of stratified nodal integral affine structures on the disk induced by its combinatorial data, then $\mathbf{I}\circ \mathbf{R}_{\mathfrak{f}}$ does not depend on the choice of a framing $\mathfrak{f}$.

\begin{figure}[h]
		\includegraphics*[width=\linewidth]{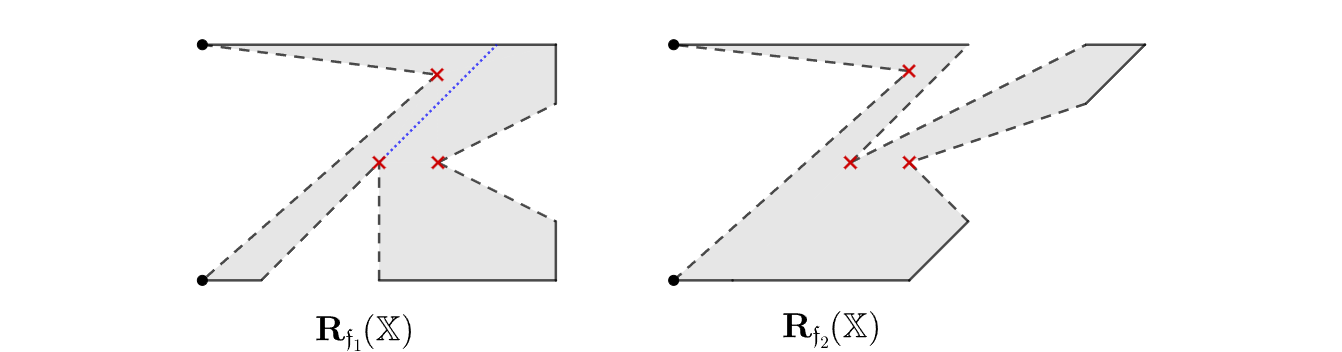}
 \caption{Two almost toric presentations of $\XX$, associated to the framings $\mathfrak{f}_1$ and $\mathfrak{f}_2$, are related by a branch move.}\label{fig:dependframing}
	\end{figure}
    
    \item {\bf (Comparison with \cite{Engel2021})} Engel and Friedman (\cite[Section 5]{Engel2021}) constructs ATF in settings of algebraic toric geometry, with the purpose towards realizing a big and nef but non-ample divisor $\lambda\in\mathcal{B}_{\text{gen}}\cap \Lambda$ as the monodromy invariant of a Type \RN{3} degeneration of holomorphic anticanonical pairs. Such a $\lambda$ corresponds to a form which is symplectic on the divisor complement but degenerate along the divisor itself. In particular, this implies that all edges in the base diagram are fully bitten, so that all components of the divisor are mapped to a single point. In other words, this only describes an ATF on the divisor complement. This suffices for their purpose, since an all-edges-fully-bitten diagram corresponds to the dual
complex of the central fiber in the Type \RN{3} degeneration. Note that although \cite[Theorem 5.4]{Engel2021} is stated for ample divisors, the effect of full bites on the edge with some $b_{ij}=0$ in their notation, is not addressed. There are additional subtleties such as the passage from $\QQ$-divisor to $\RR$-divisor, the situation in which the packing point $O$ lies on the boundary of lattice polygon arising from their use of the curve cone theorem and the canonicity of the choice of ATF presentations. We recommend \cite{MN24} for further discussions regarding these subtleties in the framework of almost complex geometry. 

\item {\bf (A possibly more straightforward proof of Corollary \ref{maincor}?)} One might attempt to prove Corollary \ref{maincor} directly, starting from the ATF on $\mathbb{CP}^2$ whose diagram is obtained by performing three nodal trades on the standard toric moment triangle. Suppose the perimeter of the triangle is $3$. Given positive numbers $\delta_1, \delta_2, \dots, \delta_n$ with $\sum_i \delta_i < 3$ (so that $c_1 \cdot [\omega] > 0$), one can consider the problem of packing possibly folded triangles of sizes $\delta_i$ into the diagram. More generally, by employing Markov $(a,b,c)$-triangles arising from mutations studied in \cite{Vianna1}, this approach works whenever $\delta_i \leq \frac{1}{3}$ for all $i$, since the longest edge has length $\frac{a}{bc}$ and the height is $\frac{bc}{a}$ (see \cite{brendlschlenk} and their interpretation of this quantity as the relative Gromov width). One can then take a sequence of Markov triples $(a_n, b_n, c_n)$ such that $\lim_{n \to \infty} \frac{a_n}{b_n c_n} = 3$. Nevertheless, extending this method to all general cases still and writing a clean argument remains challenging. See Figure \ref{fig:game}.

 \begin{figure*}[h]
		\includegraphics*[width=\linewidth]{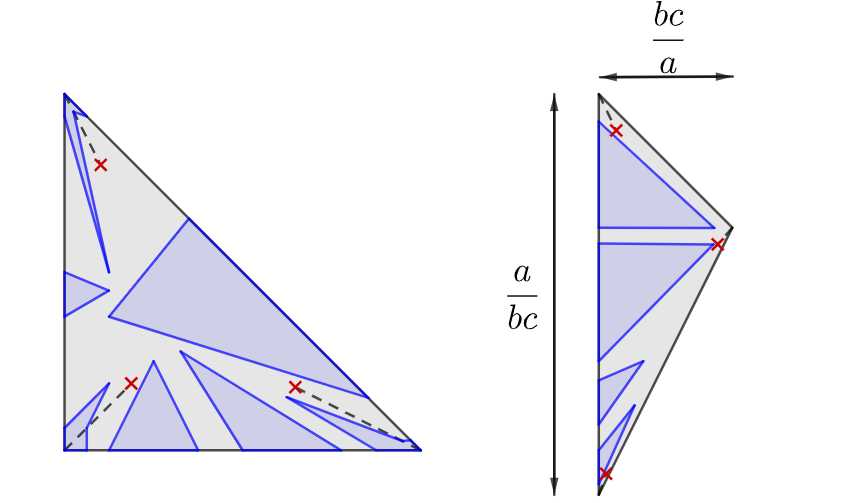}
  \caption{Pack possibly folded Symington triangles in Markov's $(a,b,c)$-triangles. }\label{fig:game}
	\end{figure*}

\item {\bf (Almost toric blowup)}  \cite[Section 5.4]{symington} explains the \textit{topological} effect of a regular Symington triangle as attaching a $(-1)$-framed $2$-handle to $S^1\times D^3$, which is equivalent to taking a connect sum with $\overline{\CC\PP}^2$. In this paper, we provide a way to interpret the \textit{symplectic} effect of a regular Symington triangle as performing a non-toric blowup on the LCY side by the well-defined quiver morphism  $\mathbf{B}:\mathbb{BD}\rightarrow\mathbb{LCY}$ (Proposition \ref{prop:bdtolcy}), which is a consequence of the symplectic Torelli Theorem \ref{prop:Torelli} and the tautness of toric pairs (Lemma \ref{lem:torictaut}). Unlike the nodal trade operation, where the uniqueness filling result of $(S^3,\xi_{\text{std}})$ and the triviality of its contact mapping class group allow for a straightforward cut-and-paste description, the ATF blowup operation seems to require a deeper understanding to maintain a surgical viewpoint. One needs to verify that the gluing contactomorphism of the boundary $(S^1\times S^2,\xi_{\text{std}})$ is trivial in the contact mapping class group, which is known to be $\ZZ\oplus\ZZ_2$  (\cite{contactmappingclassgroup})\footnote{This was communicated to us by Jonny Evans.}.

\end{itemize}

\bibliographystyle{amsalpha}
	\bibliography{main}{}
\end{document}